\documentclass[11pt, draftcls, onecolumn]{IEEEtran}

\makeatletter
\def\ps@headings{%
\def\@oddhead{\mbox{}\scriptsize\rightmark \hfil \thepage}%
\def\@evenhead{\scriptsize\thepage \hfil \leftmark\mbox{}}%
\def\@oddfoot{}%
\def\@evenfoot{}}
\makeatother
\usepackage{amsmath}
\usepackage{amsthm}
\usepackage{amsfonts}
\usepackage{amssymb}
\usepackage{url}
\usepackage{graphicx}
\usepackage{hyperref}

\usepackage{bm}

\usepackage[square, comma, sort&compress, numbers]{natbib}

\theoremstyle{plain}\newtheorem{definition}{Definition}
\theoremstyle{plain}\newtheorem{theorem}{Theorem}
\theoremstyle{plain}\newtheorem{lemma}{Lemma}
\theoremstyle{plain}\newtheorem{remark}{Remark}
\theoremstyle{plain}
\theoremstyle{plain}\newtheorem{corollary}{Corollary}

\begin{document}

\title{Tensor Matched Subspace Detection}

\author{Cuiping Li, Xiao-Yang Liu, and Yue Sun,}

\maketitle

\begin{abstract}

The problem of testing whether a signal lies within a given subspace, also named matched subspace detection, has been well studied when the signal is represented as a vector. However, the matched subspace detection methods based on vectors can not be applied to the situations that signals are naturally represented as multi-dimensional data arrays or tensors.
Considering that tensor subspaces and orthogonal projections onto theses subspaces are well defined in recently proposed transform-based tensor model, which motivates us to investigate the problem of matched subspace detection in high dimensional case. In this paper, we propose an approach for tensor matched subspace detection based on the transform-based tensor model with tubal-sampling and elementwise-sampling, respectively. First, we construct estimators based on tubal-sampling and elementwise-sampling to estimate the energy of a signal outside a given subspace of a third-order tensor and then give the probability bounds of our estimators, which show that our estimators work effectively when the sample size is greater than a constant. Secondly, the detectors both for noiseless data and noisy data are given, and the corresponding detection performance analyses are also provided. Finally, based on discrete Fourier transform (DFT) and discrete cosine transform (DCT), the performance of our estimators and detectors are evaluated by several simulations, and simulation results verify the effectiveness of our approach.

\end{abstract}


\section{Introduction}

In signal processing and big data analysis, testing whether a signal lies in a subspace is an important problem, which arises in a variety of applications, such as learning the column subspace of a matrix from incomplete data \cite{AKAS2013}, subspace clustering or identification with missing data \cite{DBLP2011,Pimentel2015}, shape detection and reconstruction from raw light detection and ranging (LiDAR) data \cite{SD2017}, image subspace representation \cite{SubRe2014}, low-complexity MIMO detection \cite{MIMO2016, MIMO2017}, tensor subspace modeling under adaptive sampling \cite{rf2016,XYL2016ICASSP}, and so on.

The problem of matched subspace detection is challenging due to three factors: 1) in cases such as Internet of Things (IoT) system \cite{8068230}, we can only obtain a data with high loss rate; 2) there is measurement noise in an observed signal; 3) existing representations of signals have limitations. Missing data will increase the difficulty of tensor matched subspace detection, and the presence of measurement noise may lead to erroneous decision. Moreover, existing mathematical models used to model the signal, such as vectors, may lead to the loss of the information of the signal, since the original structure of the signal is destroyed during modeling the signal as a mathematical model. Works for matched subspace detection in \cite{LB1994,Paredes2009,Mcwhorter2001Matched,IOPORT01742117,Balzano2010,Balzano2017,kernel2014,Martin2012} modeled a signal as a vector. However, with the developing of big data, signals can be naturally represented as multi-dimensional data arrays or tensors. When a multi-dimensional data array is represented as a vector, some information, such as the structure information between entries, will loose. Therefore, it is urgent to propose a new method for the problem of matched subspace detection based on multi-dimensional data arrays or tensors.

Tensors, as multi-dimensional modeling tools, have wide applications in signal processing \cite{Baraniuk2009,Cichocki2015,Zhang2017}, and representing a signal as a tensor can reserve more information of the original signal than representing it as a vector, for a second-order or higher-order tensor has more dimensions to describe the signal than a vector. based on the recently proposed transform-based tensor model \cite{XYL2017TSP,TTPRO2015}, a third-order tensor can be viewed as a matrix with tubes as its entries, and be treated as linear operators over the set of second-order tensors. Moreover, we have similar definitions of tensor subspace and the respective orthogonal projection in the transform-based tensor model.  Hence, the methods in \cite{Balzano2010,Balzano2017,kernel2014,Martin2012} can be extended to tensor subspaces.

In this paper, we propose a method for the problem matched subspace detection based on transform-based tensor model, called tensor matched subspace detection, and we can utilize more information of the signal than conventional methods.
First, we construct the estimators with tubal-sampling and elementwise-sampling respectively, aiming at estimating the energy of a signal outside a subspace (also called \textit{residual energy} in statistics) based on the sample. When a signal lies in the subspace, the energy of this signal outside the subspace is zero, then the energy estimated by the estimator based on the sample is also zero, but not vice versa. Secondly, bounds of our estimators are given, which show our estimators can work efficiently when the sample size is slightly than $r$ for tubal-sampling and $r\times n_3$ for elementwise-sampling, where $r$ is the dimension of the subspace. Then, the problem of tensor matched subspace detection is modeled as a binary hypothesis test with the hypotheses, $\mathcal{H}_0:$ the signal lies in the subspace, while $\mathcal{H}_1:$ the signal not lies in the subspace. With the residual energy as test statistics, the detection is given directly in the noiseless case, and for the noisy case, the constant false alarm rate (CFAR) test is made. Finally, based on discrete Fourier transform (DFT) and discrete cosine transform (DCT), our estimators and methods for tensor matched subspace detection are evaluated by corresponding experiments.

The remainder of this paper is organized as follows.  In Section \ref{sec:algebraic}, the transform-based tensor model and the problem statement are given. Then, we construct the estimators and present two theorems which give quantitative bounds on our estimators in Section \ref{sec:theorem}. The detections both with nose and without noise are given in Section \ref{sec:test}. Section \ref{sec:experiment} presents numerical experiments. Finally, Section \ref{sec:conclusion} concludes the paper.

\section{Notations and Problem Statement}
\label{sec:algebraic}

We first introduce the notations and the transform-based tensor model. Then, we formulate the problem of tensor matched subspace detection.

\subsection{Notations}
\label{ssec:natation}

Scalars are denoted by lowercase letters, e.g., $a\in\mathbb{R}$; vectors are denoted by boldface lowercase letters, e.g., $\bm{a}\in\mathbb{R}^n$; matrices are denoted by boldface capital letters, e.g., $\bm{A}\in\mathbb{R}^{m\times n}$; and third-order tensors are denoted by calligraphic letters, i.e., $\mathcal{A}\in \mathbb{R}^{n_1\times n_2\times n_3}$. The transpose of a vector or a matrix is denoted with a superscript $^H$, and the transpose of a third-order tensor is denoted with a superscript $^\dag $. We use $[n]$ to denote the index set $\{1,2,\ldots,n\}$, $[n_1]\times[n_2]$ to denote the set $\{(1,1),(1,2),\ldots,(1,n_2),(2,1),\ldots,(n_1,n_2)\}$, and $[n+m]-[n]$ to denote the set $\{n+1,\ldots,n+m\} $.

The $i$-th element of a vector $\bm{a} $ is $\bm{a}_i$, the $(i,j)$-th element of a matrix $\bm{A} $ is $\bm{A}_{i,j}$ or $\bm{A}(i,j) $, and similarly for third-order tensors $\mathcal{A}$, the $(i,j,k)$-th element is $\mathcal{A}_{i,j,k} $ or $\mathcal{A}(i,j,k)$ .  For a third-order tensor \(\mathcal{A}\in \mathbb{R}^{n_1 \times n_2 \times n_3}\), a tube of $\mathcal{A} $ is defined by fixing all indices but one, while a slice of $\mathcal{A} $ defined by fixing all but two indices. We use $\mathcal{A}(:,j,k) $, $\mathcal{A}(i,:,k) $, $\mathcal{A}(i,j,:) $ to denote mode-1, mode-2, mode-3 tubes of $\mathcal{A} $, and $\mathcal{A}(:,:,k) $, $\mathcal{A}(:,j,:) $, $\mathcal{A}(i,:,:) $ to denote  the frontal, lateral, and horizontal slices of $\mathcal{A} $. $\mathcal{A}(i,:,:) $ and $\mathcal{A}(:,j,:) $ are also called tensor row and tensor column. For easy representation, we use $\mathcal{A}^{(k)} $ to denote $\mathcal{A}(:,:,k) $, and $\mathcal{A}_j $ to denote $\mathcal{A}(:,j,:) $.

For a vector $\bm{a}\in\mathbb{R}^n $, the $\ell_2 $-norm is $\|\bm{a}\|_2= \sqrt{\sum_{i\in[n]}|\bm{a}_i|^2} $, while for a matrix $\bm{A}\in\mathbb{R}^{n_1\times n_2} $, the Frobenius norm is $\|\bm{A}\|_F=\sqrt{\sum_{i=1}^{n_1}\sum_{j=1}^{n_2}\bm{A}(i,j)^2} $, and the spectral-norm $\|\bm{A}\| $ is the largest singular value of $\bm{A} $. For a tensor $\mathcal{A}\in\mathbb{R}^{n_1\times n_2\times n_3}$, the Frobenius norm is $\|\mathcal{A}\|_F=\sqrt{\sum_{i=1}^{n_1}\sum_{j=1}^{n_2}\sum_{k=1}^{n_3}\mathcal{A}(i,j,k)^2} $. For a tensor column $\mathcal{X}\in\mathbb{R}^{n_1\times 1\times n_3} $, we define $\ell_{\infty^*} $-norm as $\|\mathcal{X}\|_{\infty^*}= \max\limits_i\|\mathcal{X}(i,1,:)\|_2 $, and $\ell_{\infty} $-norm as $\|\mathcal{X}\|_\infty=\max\limits_{i,~k} |\mathcal{X}(i,1,k)|$.

For a tube $\bm{a}\in\mathbb{R}^{1\times 1\times n_3} $ and a given linear transform $\mathcal{L}$,
\begin{equation}\label{def:trans}
  \mathcal{L}(\bm{a})(1,1,i)=(\bm{M}\text{vec}(\bm{a}))_i,
\end{equation}
where $\text{vec}(\cdot) $ is the vector representation of a tube, and  $\bm{M}$ is the matrix decided by the transform $\mathcal{L}$.
For a tube $\bm{a}\in\mathbb{R}^{1\times 1\times n_3} $, we have $\|\mathcal{L}(\bm{a})\|_F= c\|\bm{a}\|_F $, where $c$ is a constant and $0<c\leq\|\bm{M}\|$.

\subsection{Transform-based Tensor Model}
\label{ssec:framework}

In order to introduce the definition of $\mathcal{L}$-product, we first introduce the tube multiplication. Given an invertible discrete transform $\mathcal{L}:~\mathbb{R}^{1\times 1\times n_3} \rightarrow \mathbb{R}^{1\times 1\times n_3} $, the elementwise multiplication $\circ $, and $\bm{a},~\bm{b}\in\mathbb{R}^{1\times 1\times n_3} $, the tube multiplication of $\bm{a}$ and $\bm{b}$ is defined as
\begin{equation*}
  \bm{a}\bullet\bm{b}=\mathcal{L}^{-1}(\mathcal{L}(\bm{a})\circ \mathcal{L}(\bm{b})),
\end{equation*}
where $\mathcal{L}^{-1} $ is the inverse of $\mathcal{L} $ \cite{XYL2017TSP}.

\begin{definition}[Tensor product: $\mathcal{L}$-product~\cite{XYL2017TSP}]
The $\mathcal{L} $-product $\mathcal{C}=\mathcal{A}\bullet \mathcal{B}$ of $\mathcal{A}\in\mathbb{R}^{n_1\times n_2\times n_3}$ and $\mathcal{B}\in\mathbb{R}^{n_2\times n_4\times n_3}$ is a tensor of size $n_1\times n_4\times n_3$, with $\mathcal{C}(i,j,:)=\sum\limits_{s=1}^{n_2}\mathcal{A}(i,s,:)\bullet\mathcal{B}(s,j,:)$, for $i\in[n_1]$ and $j\in[n_4]$.
\end{definition}

\textbf{Transform domain representation~\cite{XYL2017TSP}:}
For an invertible discrete transform $\mathcal{L}:~\mathbb{R}^{1\times 1\times n_3} \rightarrow \mathbb{R}^{1\times 1\times n_3} $, let $\mathcal{L}(\mathcal{A})\in \mathbb{R}^{n_1\times n_2\times n_3}$ denote the tensor obtained by taking the transform $\mathcal{L}$ of all the tubes along the third dimension of $\mathcal{A}\in \mathbb{R}^{n_1\times n_2\times n_3} $, i.e., for $i\in [n_1]$ and $j\in [n_2] $, $\mathcal{L}(\mathcal{A})(i,j,:)=\mathcal{L}(\mathcal{A}(i,j,:))$. Furthermore, we use $\mathcal{\overline{A}} $ to denote the block diagonal matrix of the tensor $\mathcal{L}(\mathcal{A}) $ in the transform domain, i.e.,
 \begin{equation*}
   \mathcal{\overline{A}}=\left[
                            \begin{array}{cccc}
                              \mathcal{L}(\mathcal{A})^{(1)} &  &  &  \\
                               & \mathcal{L}(\mathcal{A})^{(2)} &  &  \\
                               &  &\ddots &  \\
                               &  &  & \mathcal{L}(\mathcal{A})^{(n_3)} \\
                            \end{array}
                          \right].
 \end{equation*}

Under the transform-based tensor model, an $n_1\times n_2\times n_3 $ tensor can be viewed as an $n_1\times n_2 $ matrix of tubes that are in the third-dimension, therefore the $ \mathcal{L}$-product of two tensors can be regarded as multiplication of two matrices, expect that the multiplication of two numbers is replaced by the multiplication of two tubes. Owing to the definition of $\mathcal{L}$-product based on the discrete transform, we have the following remark that is used throughout the paper.

\begin{remark}[\cite{XYL2017TSP}]
The $\mathcal{L}$-product $\mathcal{C}=\mathcal{A}\bullet\mathcal{B} $ can be calculated in the following way:
\begin{equation*}
  \mathcal{C}=\mathcal{L}^{-1}(\mathcal{\overline{A}}~\mathcal{\overline{B}}).
\end{equation*}
\end{remark}

Motivated by the definition of t-product in \cite{Kilmer2013} and the cosine transform based product in \cite{TTPRO2015}, we introduce the $\mathcal{L} $-product based on block matrix tools. For tensor $\mathcal{A}\in\mathbb{R}^{n_1\times n_2\times n_3} $, we use $\text{lmat}(\mathcal{A})$ to denote a special structured block matrix determined by the frontal slices of $\mathcal{A} $, such that the $\mathcal{L} $-product $\mathcal{Z}=\mathcal{A}\bullet\mathcal{C} $, where $\mathcal{C}\in \mathbb{R}^{n_2\times 1\times n_3}$ and $\mathcal{Z}\in \mathbb{R}^{n_1\times 1\times n_3} $, can be represented as $\text{unfold}(\mathcal{Z})=\text{lmat}(\mathcal{A})\cdot\text{unfold}(\mathcal{C}) $, Where
\begin{equation*}
  \text{unfold}(\mathcal{A})=\left[\begin{array}{cccc}
                                     \mathcal{A}^{(1)H} & \mathcal{A}^{(2)H} & \cdots & \mathcal{A}^{(n_3)H}
                                   \end{array}
  \right]^H.
\end{equation*}
The form of the block matrix $\text{lmat}(\mathcal{A}) $ varies with the discrete transformation \cite{Kilmer2013,DCT1995,TK2005}. When the transform $\mathcal{L}$ is discrete Fourier transform, $\text{lmat}(\mathcal{A})=\text{bcirc}(\mathcal{A}) $ \cite{Kilmer2013}, where $\text{bcirc}(\cdot) $ is the operation that converts a third order tensor into a block circular matrix, i.e.,
\begin{equation}\label{equal:bcirc}
 \text{bcirc}(\mathcal{A})=\left[
                       \begin{array}{cccc}
                        \mathcal{A}^{(1)} & \mathcal{A}^{(k)} & \cdots & \mathcal{A}^{(2)} \\
                         \mathcal{A}^{(2)} & \mathcal{A}^{(1)} & \cdots & \mathcal{A}^{(3)} \\
                         \vdots & \vdots & \ddots & \vdots \\
                         \mathcal{A}^{(k)} & \mathcal{A}^{(k-1)} & \cdots & \mathcal{A}^{(1)} \\
                       \end{array}
                     \right].
\end{equation}
When the transform $\mathcal{L}$ is discrete cosine transform, $\text{lmat}(\mathcal{A})= ((\bm{I}_{n_3}+\bm{Z}_{n_3})\otimes \bm{I}_{n_1})^{-1}(\bm{T}+\bm{H})((\bm{I}_{n_3}+ \bm{Z}_{n_3})\otimes \bm{I}_{n_2}) $, where $\otimes$ is the Kronecker product \cite{Baraniuk2009,TTPRO2015}, $\bm{I}_{n_i}$ denotes $n_i\times n_i $, $i\in[3] $, identity matrix, $\bm{Z}_{n_3} $ is the $n_3\times n_3 $ circular upshift matrix  as follows
 \begin{equation}\label{equal:Z}
  \bm{Z}_{n_3}=\left[
           \begin{array}{ccccc}
             0 & 1 & 0 & \cdots & 0 \\
             0 & 0 & 1 & \ddots & \vdots \\
             \vdots & \vdots & \ddots & \ddots & 0 \\
             0 & 0 & \cdots & 0 & 1 \\
             0 & 0 & \cdots & 0 & 0 \\
           \end{array}
         \right],
\end{equation}
and $\bm{T}+\bm{H}$ is the following $n_1n_3\times n_2n_3 $ block Toeplitz-plus-Hankel matrix \cite{DCT1995,TK2005,TTPRO2015}
 \begin{eqnarray}\label{equa:thmatrix}
   \bm{T}+\bm{H}&\!\!\!\!\!=\!\!\!\!\!&\left[
                       \begin{array}{cccc}
                        \mathcal{A}^{(1)} & \mathcal{A}^{(2)} & \cdots & \mathcal{A}^{(k)} \\
                         \mathcal{A}^{(2)} & \mathcal{A}^{(1)} & \cdots & \mathcal{A}^{(k-1)} \\
                         \vdots & \vdots & \ddots & \vdots \\
                         \mathcal{A}^{(k)} & \mathcal{A}^{(k-1)} & \cdots & \mathcal{A}^{(1)} \\
                       \end{array}
                     \right]+
     \left[
                       \begin{array}{cccc}
                        \mathcal{A}^{(2)} & \cdots & \mathcal{A}^{(k)}&\bm{0} \\
                         \vdots & \vdots &   & \mathcal{A}^{(k)} \\
                         \mathcal{A}^{(k)} & \bm{0} & \cdots & \vdots \\
                         \bm{0} & \mathcal{A}^{(k)} & \cdots & \mathcal{A}^{(2)} \\
                       \end{array}
                     \right].
 \end{eqnarray}

\begin{definition}[Tensor transpose~\cite{XYL2017TSP,TTPRO2015}]
Let $\mathcal{A}\in\mathbb{R}^{n_1\times n_2\times n_3 }$, then the transpose $\mathcal{A}^{\dag}\in\mathbb{R}^{n_2\times n_1\times n_3} $ is such that $\mathcal{L}(\mathcal{A}^{\dag})^{(i)}=(\mathcal{L}(\mathcal{A})^{(i)})^H$, $i\in[n_3] $.
\end{definition}

The transpose of $\mathcal{A}$ can be obtained by taking the inverse transform of the tensor whose $i$-th frontal slice is $(\mathcal{L}(\mathcal{A})^{(i)})^H$, $i\in[n_3]$, and the multiplication reversal property of the transpose holds \cite{XYL2017TSP,TTPRO2015}, i.e. $(\mathcal{A}\bullet\mathcal{B})^{\dag}=\mathcal{B}^{\dag}\bullet\mathcal{A}^{\dag}$.

\begin{definition}[$\mathcal{L}$-diagonal tensor \cite{rf2016}]
A tensor is called $\mathcal{L}$-diagonal tensor if each frontal slice of the tensor is a diagonal matrix.
\end{definition}

Let $\bm{e}=\mathcal{L}^{-1}(\bm{1})$, where $\bm{1}\in\mathbb{R}^{1\times 1\times n_3} $ denotes a tube of length $n_3$ with all entries equal to $\bm{1}$, and $\bm{e}$ is the multiplicative unity for the tube multiplication \cite{XYL2017TSP}. The multiplicative unity $\bm{e}$ plays a similar role in tensor space as $1$ in vector space.

\begin{definition}[Identity tensor~\cite{XYL2017TSP}]
The identity $\mathcal{I}\in\mathbb{R}^{m\times m\times n_3} $ is an $\mathcal{L}$-diagonal square tensor with $\bm{e}$'s on the main diagonal and zeros elsewhere, i.e., $\mathcal{I}(i,i,:)=\bm{e} $ for $i\in[n_3] $, where all other tubes $0$'s.
\end{definition}

A square tensor $\mathcal{A}\in\mathbb{R}^{m\times m\times n_3}$ is invertible if there exists a tensor $\mathcal{A}^{-1}\in\mathbb{R}^{m\times m\times n_3}$ such that $\mathcal{A}\bullet\mathcal{A}^{-1}=\mathcal{A}^{-1}\bullet\mathcal{A} =\mathcal{I}$ \cite{XYL2017TSP}. Moreover, $\mathcal{A}$ is $\mathcal{L}$-orthogonal, if $\mathcal{A}^{\dag}\bullet\mathcal{A}= \mathcal{A}\bullet\mathcal{A}^{\dag}=\mathcal{I} $ \cite{XYL2017TSP}.

\begin{definition}[$\mathcal{L} $-SVD \cite{XYL2017TSP}]
The $\mathcal{L} $-SVD of $ \mathcal{A}\in\mathbb{R}^{n_1\times n_2\times n_3}$ is given by $\mathcal{A}=\mathcal{U}\bullet\Sigma\bullet\mathcal{V}^{\dagger} $, where $\mathcal{U}$ and $\mathcal{V}$ are $\mathcal{L}$-orthogonal tensors of size $n_1\times n_1\times n_3$ and $n_2\times n_2\times n_3$ respectively, and $\Sigma$ is a $\mathcal{L}$-diagonal tensor of size $n_1\times n_2\times n_3$.
\end{definition}

The $\mathcal{L}$-SVD of $\mathcal{A}$ can be derived from individual matrix SVD in transform space. That is, $\mathcal{L}(\mathcal{A})^{(i)} =\mathcal{L}(\mathcal{U})^{(i)}\mathcal{L}(\Sigma)^{(i)} (\mathcal{L}(\mathcal{V})^{(i)})^H $. Then the number of non-zero tubes of $\Sigma$ is called the $\mathcal{L}$-rank of $\mathcal{A} $.

\begin{definition}[Tensor-column subspace~\cite{XYL2017TSP}]
 Let $\mathcal{A} $ be an $n_1 \times n_2 \times n_3 $ tensor with  $\mathcal{L}$-rank of $r $ $(0<r\leq\min\{n_1, n_2\})$, then the $r$-dimensional tensor-column subspace $\mathcal{S}$ spanned by the columns of \(\mathcal{A}\) is defined as
\begin{equation*}
  \mathcal{S}=\{\mathcal{X}|\mathcal{X}={\mathcal {A}_1\bullet\bm{c}_1+\mathcal {A}_2\bullet\bm{c}_2+\cdots+\mathcal {A}_{n_2}\bullet\bm{c}_{n_2}}\}
\end{equation*}
where \(\bm{c}_j\), \(j\in [n_2]\), are arbitrary tubes of length \(n_3\).
\end{definition}

\begin{remark}
Let $\mathcal{S}$ be spanned by the columns of $\mathcal{A}\in\mathbb{R}^{n_1\times n_2\times n_3} $, then $\mathcal{P}\triangleq \mathcal{A}\bullet(\mathcal{A}^{\dagger} \bullet \mathcal{A})^{-1}\bullet\mathcal{A}^{\dagger} $ is an orthogonal projection onto $\mathcal{S}$ if $\mathcal{A}^{\dagger} \bullet \mathcal{A}$ is invertible.
\end{remark}

\begin{definition}
Let $\mathcal{P}$ be the orthogonal projection onto an r-dimensional subspace $\mathcal{S}\subset\mathbb{R}^{n_1\times 1\times n_3} $, then the coherence of $\mathcal{S}$ is defined as
 \begin{equation*}
   \mu(\mathcal{S})\triangleq \frac{n_1}{r}\max_j\left\|\mathcal{P}\bullet \mathcal{E}_j\right\|_F^2,
 \end{equation*}
 where $\mathcal{E}_j $ is the $n_1\times 1\times n_3 $ tensor basis with $\mathcal{E}(j,1,:)=\bm{e} $ and zeros elsewhere.
\end{definition}

Assume the subspace $ \mathcal{S}$ is spanned by the columns of $\mathcal{A}\in\mathbb{R}^{n_1\times n_2\times n_3} $. Note that $1\leq\mu(\mathcal{S})\leq\frac{n_1}{r} $, then with low $\mu(\mathcal{S}) $, each tube of $\mathcal{A} $ carries approximately same amount of information \cite{Zhang2017}.

\subsection{Problem Formulation}
\label{ssec:problem}

Let $\mathcal{S}$ be a given $r$-dimensional subspace in $\mathbb{R}^{n_1 \times 1\times n_3}$ $(r\ll n_1 )$ spanned by the columns of a third a third-order tensor $\mathcal{U}\in\mathbb{R}^{n_1\times n_2\times n_3} $, and $\mathcal{T}\in \mathbb{R}^{n_1 \times 1\times n_3} $ denotes a signal with its entries are sampled with replacement. The problem of tensor matched subspace detection can be modeled as a binary hypothesis test with hypotheses:
\begin{equation}\label{equal:hypo}
  \left\{\begin{array}{ccc}
    \mathcal{H}_{0}  & : & \mathcal{T}\in \mathcal{S}; \\
   \mathcal{H}_{1} & : & \mathcal{T}\notin \mathcal{S}.
  \end{array}
  \right.
\end{equation}

Here, we consider two types of sampling: tubal-sampling and elementwise-sampling, as showed in Fig. \ref{fig:sampling}. We use $\Omega $ to denote the set of the index of samples, $|\Omega|$ to denote the cardinality of $\Omega $, and $\mathcal{T}_{\Omega} $ to denote the corresponding sampling signal of $\mathcal{T} $. Then the definitions of tubal-sampling and elementwise-sampling are:

\textbf{Tubal-sampling:} $\Omega \subset [n_1] $, and $\mathcal{T}_{\Omega} $ is a tensor of $|\Omega|\times 1\times n_3 $ with its tubes $\mathcal{T}_\Omega(i,1,:)=\mathcal{T}(\Omega(i),1,:) $.

\textbf{Elementwise-sampling:} $\Omega \subset [n_1]\times[n_3] $, and $\mathcal{T}_{\Omega} $ is a tensor of $n_1\times 1\times n_3 $ with its entries $\mathcal{T}_\Omega(i,1,j)=\mathcal{T}(i,1,j) $ if $(i,j)\in\Omega $ and zero if $(i,j)\notin\Omega $.

\begin{figure}[t]

\begin{minipage}[b]{1\linewidth}
  \centering
  \centerline{\includegraphics[width=5cm]{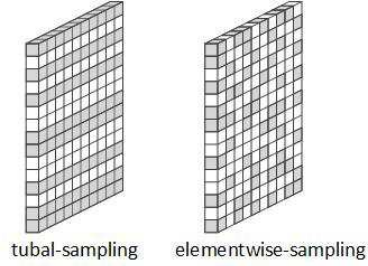}}
\end{minipage}
\caption{An illustration of tubal-sampling and elementwise-sampling patterns.}
\label{fig:sampling}
\end{figure}

Let $\mathcal{P} $ be the orthogonal projection onto $\mathcal{S} $, and $\mathcal{P}= \mathcal{U}\bullet(\mathcal{U}^{\dagger} \bullet \mathcal{U})^{-1}\bullet\mathcal{U}^{\dagger}  $. We use $\text{t}(\cdot) $ to denote the energy of a signal outside a given subspace. Then when the entries of $\mathcal{T} $ are fully observed, the test statistic can be constructed as
\begin{equation}
  \text{t}(\mathcal{T})=\|\mathcal{T}-\mathcal{P}\bullet\mathcal{T}\|_F^2 \overset{\mathcal{H}_{1}}{\underset{\mathcal{H}_{0}}{ \gtrless}} \eta.
\end{equation}
When $\mathcal{T}\in\mathcal{S} $, we have $\text{t}(\mathcal{T})\leq \eta $, and $\text{t}(\mathcal{T})>\eta $ when $\mathcal{T}\notin\mathcal{S} $. In the noiseless case, $\eta=0 $.

In practice, for high-dimensional applications, it is prohibitive or impossible to measure $\mathcal{T} $ completely, and we can only obtain a sampling signal $\mathcal{T}_\Omega $, so we can not calculate the energy of $\mathcal{T} $ outside the subspace $\mathcal{S} $ directly. Therefore, we should construct a new estimator $\|\mathcal{T}_\Omega-\mathcal{P}_\Omega\bullet\mathcal{T}_\Omega\|_F^2$ to estimate the energy of $\mathcal{T} $ outside the subspace $\mathcal{S} $ based on $\mathcal{T}_\Omega $ and the corresponding projection $\mathcal{P}_\Omega $. A good estimator should satisfy the following conditions (noiseless case):
\begin{itemize}
  \item  When $\mathcal{T}\in \mathcal{S} $, $\|\mathcal{T}-\mathcal{P}\bullet\mathcal{T}\|_F^2=0$, then $\|\mathcal{T}_\Omega-\mathcal{P}_\Omega\bullet\mathcal{T}_\Omega\|_F^2=0$ for arbitrary sample size $|\Omega| $.
  \item When $\mathcal{T}\notin \mathcal{S} $, $\|\mathcal{T}-\mathcal{P}\bullet\mathcal{T}\|_F^2>0$, then, as long as the sample size $|\Omega| $ is greater than a constant but much smaller than the size of $\mathcal{T} $, $\|\mathcal{T}_\Omega-\mathcal{P}_\Omega\bullet\mathcal{T}_\Omega\|_F^2>0$.
\end{itemize}

\section{Energy Estimation and Main Theorems}
\label{sec:theorem}

In this section, based on the tubal-sampling and elementwise-sampling, the estimators are constructed respectively. Then, two theorems are given to bound the estimators, which show that our estimators can work effectively when the sample size $|\Omega| $ is $O(r\log (rn_3)) $ for tubal sampling and $O(rn_3\log (rn_3)) $ for elementwise-sampling. Without loss of generality, we assume $\mathcal{U}\in\mathbb{R}^{n_1\times n_2\times n_3} $, whose columns span the subspace $\mathcal{S} $, is orthogonal, that means the dimension of $ \mathcal{S}$ is $ n_2$. For convenience the following representation, we set $m=|\Omega|$ to be the sample size.

\subsection{Energy Estimation}
\label{ssec:estim}

For tubal-sampling, the estimator can be constructed as follows. Note that $\mathcal{U} $ be an $n_1\times n_2\times n_3 $ tensor whose columns span the $n_2 $-dimensional subspace $\mathcal{S} $. We let $\mathcal{U}_{\Omega} $ be the $m\times n_2\times n_3 $ tensor organized by the horizontal slices of $\mathcal{U} $ indicated by $\Omega $, that means $\mathcal{U}_{\Omega}(i,:,:)=\mathcal{U}(\Omega(i),:,:) $. Then we define the projection $\mathcal{P}_{\Omega}= \mathcal{U}_{\Omega}\bullet(\mathcal{U}_{\Omega}^{\dagger} \bullet\mathcal{U}_{\Omega})^{-1}\bullet\mathcal{U}_{\Omega}^{\dagger} $. It follows immediately that if $\mathcal{T}\in \mathcal{S} $, $\|\mathcal{T}-\mathcal{P}\bullet\mathcal{T}\|_F^2=0$ and $\|\mathcal{T}_\Omega-\mathcal{P}_\Omega\bullet\mathcal{T}_\Omega\|_F^2=0$. However, it is possible that $\|\mathcal{T}_\Omega-\mathcal{P}_\Omega\bullet\mathcal{T}_\Omega\|_F^2=0$ even if $\|\mathcal{T}-\mathcal{P}\bullet\mathcal{T}\|_F^2>0$ when the sample size $m<n_2$. One of our main theorems show that if $m $ is just slightly grater than $ n_2 $, then with high probability $\|\mathcal{T}_\Omega-\mathcal{P}_\Omega\bullet\mathcal{T}_\Omega\|_F^2$ is very close to $\frac{m}{n_1}\|\mathcal{T}-\mathcal{P}\bullet\mathcal{T}\|_F^2 $.

For elementwise-sampling, the subspace $\mathcal{S}\subset{\mathbb{R}^{n_1\times 1\times n_3}} $ should be mapped into a vector subspace $\bm{S}\subset\mathbb{R}^{n_1n_3} $, and $\text{unfold}(\mathcal{T})\in \bm{S} $ for all $\mathcal{T}\in\mathcal{S} $. Let the vector subspace $\bm{S}$ be spanned by the columns of $\text{lmat}(\mathcal{U}) $, then for all $\mathcal{T}\in\mathcal{S} $, $\text{unfold}(\mathcal{T})\in \bm{S} $. However, when $\mathcal{T}\in\mathcal{S}^\bot $, $\text{unfold}(\mathcal{T})\notin \bm{S}^\bot $, where $\mathcal{S}^\bot$ is the orthogonal subspace of $\mathcal{S} $ and $\bm{S}^\bot$ is the orthogonal subspace of $\bm{S} $. Let $\mathcal{T}=\mathcal{X}+\mathcal{Y}$, where $\mathcal{X}\in\mathcal{S} $ and $\mathcal{Y}\in\mathcal{S}^\bot $. Then we use $\theta $ to denote the principle angle between $\text{unfold}(\mathcal{Y}) $ and $\bm{S}^\bot $, which is defined as follows
\begin{equation}\label{equa:angle}
 \cos(\theta)=\frac{|\langle \text{unfold}(\mathcal{Y}), \bm{P}_{\bm{S}^\bot}\text{unfold}(\mathcal{Y})\rangle|}{\|\text{unfold}(\mathcal{Y})\|_2\|\bm{P}_{\bm{S}^\bot}\text{unfold}(\mathcal{Y})\|_2}
=\frac{\|\bm{P}_{\bm{S}^\bot}\text{unfold}(\mathcal{Y})\|_2}{\|\text{unfold}(\mathcal{Y})\|_2},
\end{equation}
where $\bm{P}_{\bm{S}^\bot} $ is the orthogonal projection onto $\bm{S}^\bot $, $\langle~ ,~\rangle $ denotes the inner product of two vectors, and $|\cdot|$ denotes the absolute value.

For elementwise-sampling, the estimator can be constructed as follows. As defined in Section \ref{ssec:problem}, the sampling signal $\mathcal{T}_{\Omega } $ satisfies
\begin{equation}
  \mathcal{T}_{\Omega } (i,1,j)=\left\{
                               \begin{array}{ll}
                                 \mathcal{T}(i,1,j), & \hbox{$(i,j)\in \Omega $,} \\
                                 0, & \hbox{otherwise.}
                               \end{array}
                             \right.
\end{equation}
Let \(\bm{t}_{\Omega }=\text{unfold}(\mathcal{T}_{\Omega })\), $\bm{U}=\text{lmat}(\mathcal{U}) $, and  \(\bm{P}_{\Omega }=\bm{U}_{\Omega }(\bm{U}_{\Omega }^{H} \bm{U}_{\Omega })^{-1}\bm{U}_{\Omega }^{H}\) be the projection,  where $\bm{U}_{\Omega }\in\mathbb{R}^{n_1\times n_2\times n_3}$ satisfies
\begin{equation}
  \bm{U}_{\Omega }((j-1)n_1+i,:)=\left\{
                                        \begin{array}{ll}
                                          \bm{U}((j-1)n_1+i,:), & \hbox{$(i,j)\in \Omega $,} \\
                                          \bm{0}, & \hbox{otherwise.}
                                        \end{array}
                                      \right.
\end{equation}
Then if $\mathcal{T}\in \mathcal{S} $, $\|\mathcal{T}-\mathcal{P}\bullet\mathcal{T}\|_F^2=0$ and $\|\bm{t}_\Omega-\bm{P}_\Omega\bm{t}_\Omega\|_2^2=0$. However, it is possible that $\|\bm{t}_\Omega-\bm{P}_\Omega\bm{t}_\Omega\|_2^2=0$ even if $\|\mathcal{T}-\mathcal{P}\bullet\mathcal{T}\|_F^2>0$ when the sample size $m<n_2n_3$. One of our main theorems show that if $m $ is just slightly grater than $ n_2n_3 $, then with high probability $\|\bm{t}_\Omega-\bm{P}_\Omega\bm{t}_\Omega\|_2^2$ is very close to $\frac{m}{n_1n_3}\cos^2(\theta)\|\mathcal{T}-\mathcal{P}\bullet\mathcal{T}\|_F^2 $.

\subsection{Main Theorem with Tubal-sampling}
\label{ssec:tubal-sampling}

Rewrite $\mathcal{T}=\mathcal{X}+\mathcal{Y}$, where $\mathcal{X}\in\mathcal{S}$ and $\mathcal{Y}\in\mathcal{S}^\bot$. Hence $\|\mathcal{T}-\mathcal{P}\bullet\mathcal{T}\|_F^2 =\|\mathcal{Y}-\mathcal{P}\bullet\mathcal{Y}\|_F^2 $ and $\|\mathcal{T}_{\Omega}-\mathcal{P}_{\Omega}\bullet\mathcal{T}_{\Omega}\|_F^2 =\|\mathcal{Y}_{\Omega}-\mathcal{P}_{\Omega}\bullet\mathcal{Y}_{\Omega}\|_F^2 $ under tubal-sampling, then we have the following theorem.

\begin{theorem}
\label{theo:tubal}

 Let $\delta>0$ and $m\geq \frac{8}{3}n_2\mu(\mathcal{S})\log(\frac{2n_2n_3}{\delta})$. Then with probability at least $1-4\delta$,
    \begin{equation}
      \frac{m(1-\alpha)-c^2n_2\mu(\mathcal{S})\frac{\beta}{(1-\gamma)}}{n_1}\left\|\mathcal{T}-\mathcal{P}\bullet\mathcal{T}\right\|_F^2\leq
    \left\|\mathcal{T}_{\Omega}-\mathcal{P}_{\Omega}\bullet\mathcal{T}_{\Omega}\right\|_F^2\leq(1+\alpha)\frac{m}{n_1}\left\|\mathcal{T}-\mathcal{P}\bullet\mathcal{T}\right\|_F^2
    \end{equation} holds, where $\alpha=\sqrt{\frac{2(n_1\|\mathcal{Y}\|_{\infty^*}^2-\|\mathcal{Y}\|_F^2)}{m\|\mathcal{Y}\|_F^2}\log(\frac{1}{\delta})}+ \frac{2(n_1\|\mathcal{Y}\|_{\infty^*}^2-\|\mathcal{Y}\|_F^2)}{3m\|\mathcal{Y}\|_F^2}\log(\frac{1}{\delta})$, $\beta=\left(1+2\sqrt{\log(\frac{1}{\delta})}\right)^2$, and $\gamma=\sqrt{\frac{8c^2n_2\mu(\mathcal{S})}{3m}\log(\frac{2n_2n_3}{\delta})}$.
\end{theorem}

In order to prove Theorem \ref{theo:tubal}, the following three Lemmas, whose proofs are provided in Appendix, are needed for the proof of Theorem \ref{theo:tubal}.

\begin{lemma}\label{tubal:lemma1}
With the same $\alpha$, $\beta$, $\gamma$ given in Theorem \ref{theo:tubal}, then
\begin{equation}
  (1-\alpha)\frac{m}{n_1}\left\|\mathcal{Y}\right\|_F^2 \leq\left\|\mathcal{Y}_{\Omega}\right\|_F^2 \leq(1+\alpha)\frac{m}{n_1}\left\|\mathcal{Y}\right\|_F^2
\end{equation}
holds with probability at least $1-2\delta $.
\end{lemma}

\begin{lemma}\label{tubal:lemma2}
With the same $\alpha$, $\beta$, $\gamma$ given in Theorem \ref{theo:tubal}, then
\begin{equation}
  \left\|\mathcal{U}_{\Omega}^{\dagger} \bullet\mathcal{Y}_{\Omega}\right\|_F^2 \leq \beta \frac{mn_2\mu(\mathcal{S})}{n_1^2}\left\|\mathcal{Y}\right\|_F^2
\end{equation}
holds with probability at least $1-\delta $.
\end{lemma}

\begin{lemma}\label{tubal:lemma3}
With the same $\alpha$, $\beta$, $\gamma$ given in Theorem \ref{theo:tubal}, then
\begin{equation}
 \left\|\left(\overline{\mathcal{U}_{\Omega}}^{H}~\overline{\mathcal{U}_{\Omega}} \right)^{-1}\right\|\leq \frac{n_1}{(1-\gamma)m}
\end{equation}
holds with probability at least $1-\delta $, provided that $\gamma<1$.
\end{lemma}

\begin{proof}[Proof of Theorem \ref{theo:tubal}]
Consider  $\|\mathcal{T}_{\Omega}-\mathcal{P}_{\Omega}\bullet\mathcal{T}_{\Omega}\|_F^2 =\|\mathcal{Y}_{\Omega}-\mathcal{P}_{\Omega}\bullet\mathcal{Y}_{\Omega}\|_F^2 $, and we split $\|\mathcal{Y}_{\Omega}-\mathcal{P}_{\Omega}\bullet\mathcal{Y}_{\Omega}\|_F^2 $ into three terms:
\begin{eqnarray}
  \left\|\mathcal{Y}_{\Omega}-\mathcal{P}_{\Omega}\bullet\mathcal{Y}_{\Omega}\right\|_F^2 &=& \frac{1}{c^2}\left\|\overline{\mathcal{Y}_{\Omega}}-\overline{\mathcal{P}_{\Omega}}~\overline{\mathcal{Y}_{\Omega}}\right\|_F^2\nonumber  \\
   &=& \frac{1}{c^2}\text{trace}\left(\left(\overline{\mathcal{Y}_{\Omega}}-\overline{\mathcal{P}_{\Omega}}~\overline{\mathcal{Y}_{\Omega}}\right)^H \left(\overline{\mathcal{Y}_{\Omega}}-\overline{\mathcal{P}_{\Omega}}~\overline{\mathcal{Y}_{\Omega}}\right)\right)\nonumber \\
   &=& \frac{1}{c^2}\text{trace}\left(\overline{\mathcal{Y}_{\Omega}}^H\overline{\mathcal{Y}_{\Omega}}-\overline{\mathcal{Y}_{\Omega}}^H\overline{\mathcal{P}_{\Omega}}~\overline{\mathcal{Y}_{\Omega}}\right)\nonumber \\
   &=& \|\mathcal{Y}_{\Omega}\|_F^2 - \frac{1}{c^2}\text{trace}\left(\overline{\mathcal{Y}_{\Omega}}^H \overline{\mathcal{U}_{\Omega}}(\overline{\mathcal{U}_{\Omega}}^H\overline{\mathcal{U}_{\Omega}})^{-1}\overline{\mathcal{U}_{\Omega}}^H~\overline{\mathcal{Y}_{\Omega}}\right)\nonumber \\
   &\geq & \left\|\mathcal{Y}_{\Omega}\right\|_F^2-\left\|\left(\overline{\mathcal{U}_{\Omega}}^{H}~\overline{\mathcal{U}_{\Omega}} \right)^{-1}\right\| \left\|\mathcal{U}_{\Omega}^{\dagger}\bullet\mathcal{Y}_{\Omega}\right\|_F^2,
\end{eqnarray}
where $\text{trace}(\cdot)$ denotes the trace of a matrix. Taking the union bounds of Lemma \ref{tubal:lemma1}, Lemma \ref{tubal:lemma2} and Lemma \ref{tubal:lemma3}, we have
\begin{equation}
 \frac{m(1-\alpha)-c^2n_2\mu(\mathcal{S})\frac{\beta}{(1-\gamma)}}{n_1}\left\|\mathcal{Y}\right\|_F^2\leq\left\|\mathcal{Y}_{\Omega}-\mathcal{P}_{\Omega}\bullet\mathcal{Y}_{\Omega}\right\|_F^2\leq(1+\alpha)\frac{m}{n_1}\left\|\mathcal{Y}\right\|_F^2
\end{equation}
with probability at least $1-4\delta$.
\end{proof}

\subsection{Main Theorem with Elementwise-sampling}

As described in Section \ref{ssec:estim}, the subspace $\mathcal{S} $ is mapped into the vector subspace $\bm{S} $ for elementwise-sampling. Then the coherence of $\bm{S} $ is needed. The coherence of $\bm{S} $ is defined as
\begin{equation*}
   \mu(\bm{S})\triangleq \frac{n_1}{n_2}\max_j\left\|\bm{P}_{\bm{S}} \bm{e}_j\right\|_2^2,
 \end{equation*}
where $\bm{e}_j $ is a standard basis of $\mathbb{R}^{n_1n_3} $ and $n_2n_3 $ is the dimension of $\bm{S} $.
Recall $\mathcal{T}=\mathcal{X}+\mathcal{Y} $ where $\mathcal{X}\in \mathcal{S} $ and $\mathcal{Y}\in \mathcal{S}^\bot $. Let $\bm{t}=\text{unfold}(\mathcal{T})$, and we rewrite $\bm{t}=\bm{x}+\bm{y} $, where $\bm{x}\in \bm{S} $, but $\bm{y}\in \bm{S}^\bot $. Furthermore, $\|\bm{y}\|_2=\cos(\theta)\|\mathcal{Y}\|_F $. Let $\bm{y}_{\Omega }$ be the sample of $\bm{y} $ and $\|\bm{t}-\bm{P}\bm{t}\|_2^2 =\|\bm{y}-\bm{P}\bm{y}\|_2^2 $ and $\|\bm{t}_{\Omega }-\bm{P}_{\Omega }\bm{t}_{\Omega }\|_2^2 =\|\bm{y}_{\Omega }-\bm{P}_{\Omega }\bm{y}_{\Omega }\|_2^2$ for elementwise-sampling. Then we have the following theorem.

\begin{theorem} \label{theo:element}
    Let $\delta>0$, $m\geq \frac{8}{3}n_2n_3\mu(\bm{S})\log(\frac{2n_2n_3}{\delta})$, then with probability at least $1-4\delta$
    \begin{equation}
      \frac{m(1-\alpha)-n_2n_3\mu(\bm{S})\frac{\beta}{(1-\gamma)}}{n_1n_3}\cos^2(\theta)\left\|\mathcal{T}-\mathcal{P}\bullet\mathcal{T}\right\|_F^2\leq
    \left\|\bm{t}_{\Omega }-\bm{P}_{\Omega }\bm{t}_{\Omega }\right\|_2^2 \leq(1+\alpha)\frac{m}{n_1n_3}\cos^2(\theta)\left\|\mathcal{T}-\mathcal{P}\bullet\mathcal{T}\right\|_F^2
    \end{equation} holds, where $\alpha=\sqrt{\frac{2(n_1n_3\|\bm{y}\|_\infty^2-\|\bm{y}\|_2^2)}{m\|\bm{y}\|_2^2}\log(\frac{1}{\delta})}+ \frac{2(n_1n_3\|\bm{y}\|_\infty^2-\|\bm{y}\|_2^2)}{3m\|\bm{y}\|_2^2}\log(\frac{1}{\delta})$, $\beta=\left(1+2\sqrt{\log(\frac{1}{\delta})}\right)^2$,  $\gamma=\sqrt{\frac{8n_2n_3\mu(\bm{S})}{3m}\log(\frac{2n_2n_3}{\delta})}$.
\end{theorem}

We need the following three Lemmas to prove Theorem \ref{theo:element}, and the proof of lemma \ref{element:lemma2} is provided in Appendix.

\begin{lemma}[\cite{AKAS2013}]\label{element:lemma1}
With the same $\alpha$, $\beta$, $\gamma$ given in Theorem \ref{theo:element}, then
\begin{equation}
  (1-\alpha)\frac{m}{n_1n_3}\|\bm{y}\|_2^2\leq\|\bm{y}_{\Omega }\|_2^2 \leq(1+\alpha)\frac{m}{n_1n_3}\|\bm{y}\|_2^2
\end{equation}
holds with probability at least $1-2\delta $.
\end{lemma}

\begin{lemma}\label{element:lemma2}
With the same $\alpha$, $\beta$, $\gamma$ given in Theorem \ref{theo:element}, then
\begin{equation}
   \left\|\bm{U}_{\Omega }^{H} \bm{y}_{\Omega }\right\|_2^2\leq \beta \frac{mn_2}{n_1^2n_3}\mu(\bm{S})\left\|\bm{y}\right\|_2^2
\end{equation}
holds with probability at least $1-\delta $.
\end{lemma}

\begin{lemma}[\cite{Balzano2010}]\label{element:lemma3}
With the same $\alpha$, $\beta$, $\gamma$ given in Theorem \ref{theo:element}, then
\begin{equation}
   \left\|\left(\bm{U}_{\Omega }^H \bm{U}_{\Omega }\right)^{-1}\right\|\leq \frac{n_1n_3}{(1-\gamma)m}
\end{equation}
holds with probability at least $1-\delta $, provided that $\gamma<1$.
\end{lemma}

\begin{proof}[Proof of Theorem \ref{theo:element}]
Consider $\left\|\bm{t}_{\Omega }-\bm{P}_{\Omega }\bm{t}_{\Omega }\right\|_2^2 =\left\|\bm{y}_{\Omega }-\bm{P}_{\Omega }\bm{y}_{\Omega }\right\|_2^2 $. In order to apply these three Lemmas, we split $\left\|\bm{y}_{\Omega }-\bm{P}_{\Omega }\bm{y}_{\Omega }\right\|_2^2 $ into three terms and bound each with high probability.

Assume $\bm{U}_{\Omega }^{H}\bm{U}_{\Omega } $ is invertible, and according to \cite{Balzano2010}, we have
\begin{eqnarray*}
  \left\|\bm{y}_{\Omega }-\bm{P}_{\Omega }\bm{y}_{\Omega }\right\|_2^2 &=& \left\|\bm{y}_{\Omega }\right\|_2^2-\bm{y}_{\Omega }^{H} \bm{U}_{\Omega }(\bm{U}_{\Omega }^{H} \bm{U}_{\Omega })^{-1}\bm{U}_{\Omega }^{H}\bm{y}_{\Omega } \\
   &\geq& \left\|\bm{y}_{\Omega }\right\|_2^2-\left\|(\bm{U}_{\Omega }^{H} \bm{U}_{\Omega })^{-1}\right\|\left\|\bm{U}_{\Omega }^{H} \bm{y}_{\Omega }\right\|_2^2.
\end{eqnarray*}

Combining Lemma 4, Lemma 5, and Lemma 6, we have
\begin{equation*}
  \frac{m(1-\alpha)-n_2n_3\mu(\bm{S})\frac{\beta}{(1-\gamma)}}{n_1n_3} \left\|\bm{y}\right\|_2^2\leq \left\|\bm{y}_{\Omega }-\bm{P}_{\Omega }\bm{y}_{\Omega }\right\|_2^2\leq (1+\alpha)\frac{m}{n_1n_3}\left\|\bm{y}\right\|_2^2
\end{equation*}
with probability at least $1-4\delta$.
\end{proof}

\subsection{Main Results with DFT and DCT}
\label{ssec:resdft}

When the transform $\mathcal{L} $ is DFT, $\bm{M}=\bm{F}_{n_3} $, where $\bm{F}_{n_3} $ denotes the $n_3\times n_3$ DFT matrix. For $\mathcal{A}\in\mathbb{R}^{n_1\times n_2\times n_3} $, we have $\|\mathcal{L}(\mathcal{A})\|_F=\sqrt{n_3}\|\mathcal{A}\|_F $. Moreover, $\theta=0$, that means $\text{unfold}(\mathcal{T})\in\bm{S}^\bot$ for all $\mathcal{T}\in\mathcal{S}^\bot$. Furthermore, the coherence of $\bm{S}$ and $\mathcal{S}$ are equivalent, that means $\mu(\bm{S})=\mu(\mathcal{S}) $. Thus, for the transform $\mathcal{L}$ with DFT, we have Corollary \ref{coro:tubal} for tubal-sampling and Corollary \ref{coro:element} for elementwise-sampling.

\begin{corollary}
\label{coro:tubal}

 Let $\delta>0$ and $m\geq \frac{8}{3}n_2\mu(\mathcal{S})\log(\frac{2n_2n_3}{\delta})$. Then with probability at least $1-4\delta$,
    \begin{equation}
      \frac{m(1-\alpha)-n_3n_2\mu(\mathcal{S})\frac{\beta}{(1-\gamma)}}{n_1}\left\|\mathcal{T}-\mathcal{P}\bullet\mathcal{T}\right\|_F^2\leq
    \left\|\mathcal{T}_{\Omega}-\mathcal{P}_{\Omega}\bullet\mathcal{T}_{\Omega}\right\|_F^2\leq(1+\alpha)\frac{m}{n_1}\left\|\mathcal{T}-\mathcal{P}\bullet\mathcal{T}\right\|_F^2
    \end{equation} holds, where $\alpha=\sqrt{\frac{2(n_1\|\mathcal{Y}\|_{\infty^*}^2-\|\mathcal{Y}\|_F^2)}{m\|\mathcal{Y}\|_F^2}\log(\frac{1}{\delta})}+ \frac{2(n_1\|\mathcal{Y}\|_{\infty^*}^2-\|\mathcal{Y}\|_F^2)}{3m\|\mathcal{Y}\|_F^2}\log(\frac{1}{\delta})$, $\beta=\left(1+2\sqrt{\log(\frac{1}{\delta})}\right)^2$, and $\gamma=\sqrt{\frac{8n_3n_2\mu(\mathcal{S})}{3m}\log(\frac{2n_2n_3}{\delta})}$.
\end{corollary}

\begin{corollary} \label{coro:element}
    Let $\delta>0$, $m\geq \frac{8}{3}n_2n_3\mu(\mathcal{S})\log(\frac{2n_2n_3}{\delta})$, then with probability at least $1-4\delta$
    \begin{equation}
      \frac{m(1-\alpha)-n_2n_3\mu(\mathcal{S})\frac{\beta}{(1-\gamma)}}{n_1n_3}\left\|\mathcal{T}-\mathcal{P}\bullet\mathcal{T}\right\|_F^2\leq
    \left\|\bm{t}_{\Omega }-\bm{P}_{\Omega }\bm{t}_{\Omega }\right\|_2^2 \leq(1+\alpha)\frac{m}{n_1n_3}\left\|\mathcal{T}-\mathcal{P}\bullet\mathcal{T}\right\|_F^2
    \end{equation} holds, where $\alpha=\sqrt{\frac{2(n_1n_3\|\bm{y}\|_\infty^2-\|\bm{y}\|_2^2)}{m\|\bm{y}\|_2^2}\log(\frac{1}{\delta})}+ \frac{2(n_1n_3\|\bm{y}\|_\infty^2-\|\bm{y}\|_2^2)}{3m\|\bm{y}\|_2^2}\log(\frac{1}{\delta})$, $\beta=\left(1+2\sqrt{\log(\frac{1}{\delta})}\right)^2$,  $\gamma=\sqrt{\frac{8n_2n_3\mu(\mathcal{S})}{3m}\log(\frac{2n_2n_3}{\delta})}$.
\end{corollary}

When the transformation $\mathcal{L} $ is DCT, $\bm{M}=\bm{W}^{-1}\bm{C}_{n_3}(\bm{I}_{n_3}+\bm{Z}_{n_3}) $, where $\bm{C}_{n_3}$ denotes the $n_3\times n_3$ DCT matrix, $\text{diag}(\bm{W}=\bm{C}_{n_3}(:,1)) $ is the diagonal matrix made of the first column of $\bm{C}_{n_3}$, $\bm{I}_{n_3}$ is the $n_3\times n_3$ identity matrix, and $\bm{Z}_{n_3}$ is the $n_3\times n_3$ circular upshift matrix \cite{TTPRO2015}. Then $\|\bm{M}\|=\|\bm{W}^{-1}\bm{C}_{n_3}(\bm{I}_{n_3}+\bm{Z}_{n_3})\|\leq\|\bm{W}^{-1}\|\|\bm{I}_{n_3}+\bm{Z}_{n_3}\|\leq \sqrt{n_3} $, and $0<c\leq\sqrt{n_3} $. Hence, for tubal-sampling, we have the following Corollary.

\begin{corollary}
With the same $\delta $, $m $, $\alpha $, $\beta $, $\gamma $ given in Theorem \ref{theo:tubal}, we have
\begin{equation}
      \frac{m(1-\alpha)-n_3n_2\mu(\mathcal{S})\frac{\beta}{(1-\gamma)}}{n_1}\left\|\mathcal{T}-\mathcal{P}\bullet\mathcal{T}\right\|_F^2\leq
    \left\|\mathcal{T}_{\Omega}-\mathcal{P}_{\Omega}\bullet\mathcal{T}_{\Omega}\right\|_F^2\leq(1+\alpha)\frac{m}{n_1}\left\|\mathcal{T}-\mathcal{P}\bullet\mathcal{T}\right\|_F^2
\end{equation}
with the probability at least $1-4\delta $.
\end{corollary}
Moreover, we have the same bounds of $\left\|\bm{t}_{\Omega }-\bm{P}_{\Omega }\bm{t}_{\Omega }\right\|_2^2 $ as Theorem \ref{theo:element} for elementwise-sampling.

\section{Matched Subspace Detection}
\label{sec:test}

The detection will be given in this section based on the former residual energy estimation both for noiseless data and noisy data.

\subsection{Detection with Noiseless Data}
\label{ssub:noiseless}

Recall that the hypotheses are:
\begin{equation}\label{h}
  \left\{\begin{array}{ccc}
    \mathcal{H}_{0}  & : & \mathcal{T}\in \mathcal{S}; \\
   \mathcal{H}_{1} & : & \mathcal{T}\notin \mathcal{S}.
  \end{array}
  \right.
\end{equation}

Under tubal-sampling, the test statistic is
\begin{equation}\label{de3}
  \text{t}(\mathcal{T}_{\Omega})=\|\mathcal{T}_{\Omega}- \mathcal{P}_{\Omega}\bullet\mathcal{T}_{\Omega}\|_F^2 \overset{\mathcal{H}_{1}}{\underset{\mathcal{H}_{0}}{ \gtrless}} \eta,
\end{equation}
while under elementwise-sampling
\begin{equation}\label{de1}
  \text{t}(\mathcal{T}_{\Omega})=\|\bm{t}_{\Omega}- \bm{P}_{\Omega}\bm{t}_{\Omega}\|_2^2\overset{\mathcal{H}_{1}}{\underset{\mathcal{H}_{0}}{ \gtrless}} \eta.
\end{equation}

In the noiseless case, the detection threshold $\eta=0$. For tubal-sampling, Theorem \ref{theo:tubal} shows that for $\delta>0 $, the detection probability is $ P_D=\mathbb{P} [\text{t}(\mathcal{T}_{\Omega})>0|\mathcal{H}_{1}]\geq 1-4\delta$  provided $m\geq\frac{8}{3}n_2\mu(\mathcal{S})\log(\frac{2n_2}{\delta})$. For elementwise-sampling, Theorem \ref{theo:element} the probability of detection is $ P_D=\mathbb{P} [\text{t}(\mathcal{T}_{\Omega})>0|\mathcal{H}_{1}]\geq 1-4\delta$ as long as $m\geq\frac{8}{3}n_2n_3\mu(\bm{S})\log(\frac{2n_2n_3}{\delta})$. When $\mathcal{T}\in\mathcal{S} $, $\|\mathcal{T}_{\Omega}-\mathcal{P}_{\Omega}\bullet\mathcal{T}_{\Omega}\|_F^2=0 $ and $\|\bm{t}_{\Omega}-\bm{p}_{\Omega}\bm{t}_{\Omega}\|_2^2=0$, so the false alarm probability is zero, i.e., $P_{FA}=\mathbb{P}[\text{t}(\mathcal{T}_{\Omega})>0|\mathcal{H}_{0}]=0 $ for both tubal-sampling and elementwise-sampling.

\subsection{Detection with Noisy Data}
\label{ssub:noisy}

For $\mathcal{T}\in\mathbb{R}^{n_1\times 1 \times n_3}$, assume there is Gaussian white noise $\mathcal{N}\in\mathbb{R}^{n_1\times 1 \times n_3}$ with entries $\mathcal{N}(i,1,k)\sim N(0,1),~ i\in[n_1], k\in[n_3]$ being independent. Then the test statistic can be calculated on $\mathcal{W}=\mathcal{T}_{\Omega}+\mathcal{N}_{\Omega} $, where $\mathcal{N}_{\Omega} $ is the sampling noise which is obtained by the same way as $\mathcal{T}_{\Omega} $. 

\subsubsection{Detection under Tubal-sampling}

For tubal-sampling, the test statistic is represented as
\begin{equation}\label{de4}
  \text{t}(\mathcal{W})=\|\mathcal{W}- \mathcal{P}_{\Omega}\bullet\mathcal{W}\|_F^2 \overset{\mathcal{H}_{1}}{\underset{\mathcal{H}_{0}}{ \gtrless}} \eta.
\end{equation}

Considering that
\begin{eqnarray*}
  \|\mathcal{W}- \mathcal{P}_{\Omega}\bullet \mathcal{W}\|_F^2 &=& \|\text{unfold}(\mathcal{W})- \text{lmat}(\mathcal{P}_{\Omega})\cdot \text{unfold}(\mathcal{W})\|_2^2 \\
   &=& \|(\bm{I}_{mn_3}- \text{lmat}(\mathcal{P}_{\Omega})) \cdot \text{unfold}(\mathcal{W})\|_2^2 \\
   &=&\text{unfold}(\mathcal{W})^H\cdot(\bm{I}_{mn_3}- \text{lmat}(\mathcal{P}_{\Omega}))^H \cdot(\bm{I}_{mn_3}- \text{lmat}(\mathcal{P}_{\Omega}))\cdot \text{unfold}(\mathcal{W}).
\end{eqnarray*}
Let $\bm{B}$ be an $mn_3\times mn_3 $ orthogonal matrix which converts $(\bm{I}- \text{lmat}(\mathcal{P}_{\Omega}))^H \cdot(\bm{I}_{mn_3}- \text{lmat}(\mathcal{P}_{\Omega})) $ to the diagonal form $\bm{\Lambda} =\text{diag}(\sigma_1^2,\cdots,\sigma_{mn_3}^2)= \bm{B}^H(\bm{I}_{mn_3}- \text{lmat}(\mathcal{P}_{\Omega}))^H\cdot (\bm{I}_{mn_3}- \text{lmat}(\mathcal{P}_{\Omega}))\bm{B}$, where $\sigma_i^2$, $i\in[mn_3]$ denote the singular values of $\bm{I}_{mn_3}- \text{lmat}(\mathcal{P}_{\Omega}) $. We assume the rank of $\bm{I}_{mn_3}- \text{lmat}(\mathcal{P}_{\Omega})$ is $(m-n_2)n_3 $, that means $\sigma_i^2>0 $ if $ i\in[(m-n_2)n_3]$ and zero if $i\in[mn_3]-[(m-n_2)n_3] $.
Then, $\text{t}(\mathcal{W})$ can be formulated as
\begin{equation}\label{equa:sumlam}
 \text{t}(\mathcal{W})=\sum\limits_{i=1}^{(m-n_2)n_3}\sigma_i^2\chi_1^2(\lambda_i^2),
\end{equation}
where $\chi_1^2(\lambda_i^2)$, $i\in[(m-n_2)n_3] $ are one freedom noncentral chi-squared random variables with noncentral parameter $\lambda_i^2= |\bm{B}\cdot\text{unfold}(\mathcal{T}_\Omega)(i,1)|^2 $. Moreover, $\lambda_i^2=0 $, for hypothesis $\mathcal{H}_0$.

Under tubal-sampling, the threshold $\eta=\eta_p $ is chosen to achieve a constant false alarm rate (CFAR) $P_{FA}= p $, that is
\begin{eqnarray}\label{equa:thre}
 P_{FA}&=&\mathbb{P}[\text{t}(\mathcal{W})>\eta_p|\mathcal{H}_{0}] \nonumber\\
 &=&  \mathbb{P}\left[\sum\limits_{i=1}^{(m-n_2)n_3}\sigma_i^2\chi_1^2(0)>\eta_p\right]= p.
\end{eqnarray}
Therefore, the detection probability is
\begin{eqnarray}\label{equa:pd}
 P_{D}&=&\mathbb{P}[\text{t}(\mathcal{W})>\eta_p|\mathcal{H}_{1}]\nonumber \\
 &=&  \mathbb{P}\left[\sum\limits_{i=1}^{(m-n_2)n_3}\sigma_i^2\chi_1^2(\lambda_i^2)>\eta_p\right].
\end{eqnarray}

Using the chi-square approximation derived in \cite{LIU2009853}, the approximations of Equation (\ref{equa:thre}) and (\ref{equa:pd}) is given as follows.

Let $c_k=\sum_{i=1}^{(m-n_2)n_3}\sigma_i^{2k}+\sum_{i=1}^{(m-n_2)n_3}\sigma_i^{2k}\lambda_i^{2k} $, $\mu_Q=c_1 $, $\sigma_Q= \sqrt{2c_2} $, $s_1=c_3/c_2^{\frac{3}{2}} $, and $s_2=c_4/c_2^2 $. We set $a=1/\left(s_1-\sqrt{s_1^2-s_2}\right) $ for $s_1^2>s_2$, while $a=1/s_1$ for $s_1^2\leq s_2$. Moreover, we set $\lambda_\chi^2=s_1a^3-a^2 $, and $l=a^2-2\lambda_\chi^2 $. Then we have the following.
\begin{eqnarray}\label{equa:approx}
 \mathbb{P}[\text{t}(\mathcal{W})>\eta_p] &=&\mathbb{P}\left[\sum\limits_{i=1}^{(m-n_2)n_3}\sigma_i^2\chi_1^2(\lambda_i^2)>\eta_p\right]  \nonumber\\
   &\approx&
   \mathbb{P}\left[\chi_l^2(\lambda_\chi^2)>\eta_p^*\sigma_{\chi}+\mu_{\chi}\right],
\end{eqnarray}
where $\eta_p^*=(\eta_p-\mu_Q)/\sigma_Q$, $\mu_{\chi}=1+\lambda_\chi^2 $, and $\sigma_{\chi}=\sqrt{2}a $. Therefore, we can obtain the threshold by the following
\begin{eqnarray*}
  P_{FA} &=& \mathbb{P}[\text{t}(\mathcal{W})>\eta_p|\mathcal{H}_0] \\
   &\approx& \mathbb{P}\left[\chi_l^2(\lambda_\chi^2)>\eta_p^*\sigma_{\chi}+\mu_{\chi}|\lambda_i^2=0,i\in[(m-n_2)n_3]\right]= p,
\end{eqnarray*}
and the detection probability is
\begin{eqnarray*}
  P_{D} &=& \mathbb{P}[\text{t}(\mathcal{W})>\eta_p|\mathcal{H}_1] \\
   &\approx& \mathbb{P}\left[\chi_l^2(\lambda_\chi^2)>\eta_p^*\sigma_{\chi}+\mu_{\chi}|\mathcal{H}_1\right].
\end{eqnarray*}

Consider a special case of detection problem (\ref{de4}) that $\bm{I}_{mn_3}- \text{lmat}(\mathcal{P}_{\Omega}) $ is an orthogonal projection onto an $(m-n_2)n_3 $ dimensional subspace of $\mathbb{R}^{mn_3} $, and $\sigma_i^2=1 $ if $i\in[(m-n_2)n_3]$ and zero if $i\in[mn_3]-[(m-n_2)n_3] $. Then Equation (\ref{equa:sumlam}) can be simplified as
\begin{equation*}
  \text{t}(\mathcal{W})=\chi_{(m-n_2)n_3}^2(\lambda^2),
\end{equation*}
where $\lambda^2=\sum_{i=1}^{(m-n_2)n_3}\lambda_i^2=\|(\bm{I}_{mn_3}- \text{lmat}(\mathcal{P}_{\Omega}))\cdot \text{unfold}(\mathcal{T}_\Omega)\|_2^2 $.
Then the threshold can be obtain by the following
\begin{equation*}
  P_{FA}=1- \mathbb{P}\left[\chi_{(m-n_2)n_3}^2(0)\leq\eta_p\right]= p,
\end{equation*}
and the detection probability is
\begin{equation*}
P_{D}=1- \mathbb{P}\left[\chi_{(m-n_2)n_3}^2(\lambda^2)\leq\eta_p\right].
\end{equation*}

\subsubsection{Detection under Elementwise-sampling}

For elementwise-sampling, we set $\bm{w}=\text{unfold}(\mathcal{W})=\bm{t}_\Omega+\text{unfold}(\mathcal{N}_\Omega) $, and then the test statistic can be represented as follows

\begin{equation}
  \text{t}(\mathcal{W})=\|\bm{w}- \bm{P}_{\Omega}\bm{w}\|_2^2 \overset{\mathcal{H}_{1}}{\underset{\mathcal{H}_{0}}{ \gtrless}}\eta.
\end{equation}

Considering $\text{t}(\mathcal{W})=\|\bm{w}- \bm{P}_{\Omega}\bm{w}\|_2^2= \bm{w}^H(\bm{I}_{m}- \bm{P}_{\Omega})^H(\bm{I}_{m}- \bm{P}_{\Omega})\bm{w} $, we can find an $m\times m $ orthogonal matrix $\bm{B} $ which can convert $(\bm{I}_{m}- \bm{P}_{\Omega})^H (\bm{I}_{m}- \bm{P}_{\Omega}) $ to the diagonal form $\bm{\Lambda} =\text{diag}(\sigma_1^2,\cdots,\sigma_{m}^2)= \bm{B}^H(\bm{I}_{m}- \bm{P}_{\Omega})^H (\bm{I}_{m}-\bm{P}_{\Omega})\bm{B}$, where $\sigma_i^2$, $i\in[m]$ denote the singular values of $\bm{I}_{m}- \bm{P}_{\Omega} $.

For $\bm{I}_{m}- \bm{P}_{\Omega} $ being the orthogonal projection onto an $m-n_2n_3 $ dimensional vector subspace of $\mathbb{R}^{m} $, we have $\sigma_i^2=1 $ if $i\in[m-n_2n_3] $ and zero if $i\in[m]-[m-n_2n_3] $. Then, $\text{t}(\mathcal{W})$ is a non-central $\chi^2$-distributed variable with degree of freedom $m-n_2n_3$ and non-centrality parameter $\lambda^2=\|\bm{t}_{\Omega}- \bm{P}_{\Omega}\bm{t}_{\Omega}\|_2^2 $. We can choose threshold $\eta=\eta_p $ to achieve a CFAR $P_{FA}= p $. That is
\begin{eqnarray*}
 P_{FA}&=&\mathbb{P}[\text{t}(\mathcal{W})>\eta_p|\mathcal{H}_{0}]\\
 &=&1- \mathbb{P}[\chi_{m-n_2n_3}^2(0)<\eta_p]= p,
\end{eqnarray*}
and the detection probability is
\begin{eqnarray*}
  P_D &=& \mathbb{P}[\text{t}(\mathcal{W})>\eta_p|\mathcal{H}_{1}]\\
  &=& 1-\mathbb{P} [\chi_{m-n_2n_3}^2(\lambda^2)\leq\eta_p].
\end{eqnarray*}

\section{Performance Evaluation}
\label{sec:experiment}

In this section, the performance of estimators and the corresponding detectors are evaluated by some simulations with synthetic data based on DFT and DCT. The synthetic data, generated according to the transform-based tensor model, serves as well-controlled inputs for testing and understanding our main theorems.

\subsection{Synthetic Data}

We set $n_1=50 $, $r=10 $, $n_3=50 $, and $\mathcal{S}$ is a $10$ dimensional subspace of $\mathbb{R}^{50\times 1\times 50} $. Moreover, $\mathcal{S}$ is spanned by the orthogonal columns of $\mathcal{U}\in\mathbb{R}^{50\times 10\times 50}  $.

First, we generate a tensor $\mathcal{A}$ of size $50\times 10\times 50$, and each entry follows uniformly distribute between 0 and 1. The $\mathcal{L}$-rank of $\mathcal{A}$ is $10$. Secondly, $\mathcal{A}$ is decomposed into three parts, $\mathcal{A}=\mathcal{U}\bullet\Sigma\bullet\mathcal{V}^{\dagger} $, by $\mathcal{L}$-SVD. Thirdly, $\mathcal{U}$ is divided into two parts, $\mathcal{U}=\left[\mathcal{U}_r~\widetilde{\mathcal{U}}_r\right] $, where $\mathcal{U}_r $ consists of the first $r$ columns of $\mathcal{U}$ and the rest of the columns compose $\widetilde{\mathcal{U}}_r $. Fourthly, we generate another two tensors, $\mathcal{C}_1$ of size $10\times 1\times 50 $ and $\mathcal{C}_2$ of size $40\times 1\times 50 $, respectively. All entries of $\mathcal{C}_1 $ and $\mathcal{C}_2 $ follow uniformly distribute between 0 and 1. Then, multiplying $\mathcal{U}_r $ and $\mathcal{C}_1 $ we obtain a signal in $\mathcal{S} $, and we obtain a signal in $\mathcal{S}^\bot $ while $\widetilde{\mathcal{U}}_r $ is multiplied by $\mathcal{C}_2 $. Finally, the signal is normalized, such that $\|\mathcal{T}\|_F^2=1  $.

\subsection{Performance Evaluation for Estimators}
\label{ssec:performEva}

\begin{figure}[t]

\begin{minipage}[t]{0.48\linewidth}
  \center
  \centerline{\includegraphics[width=8cm]{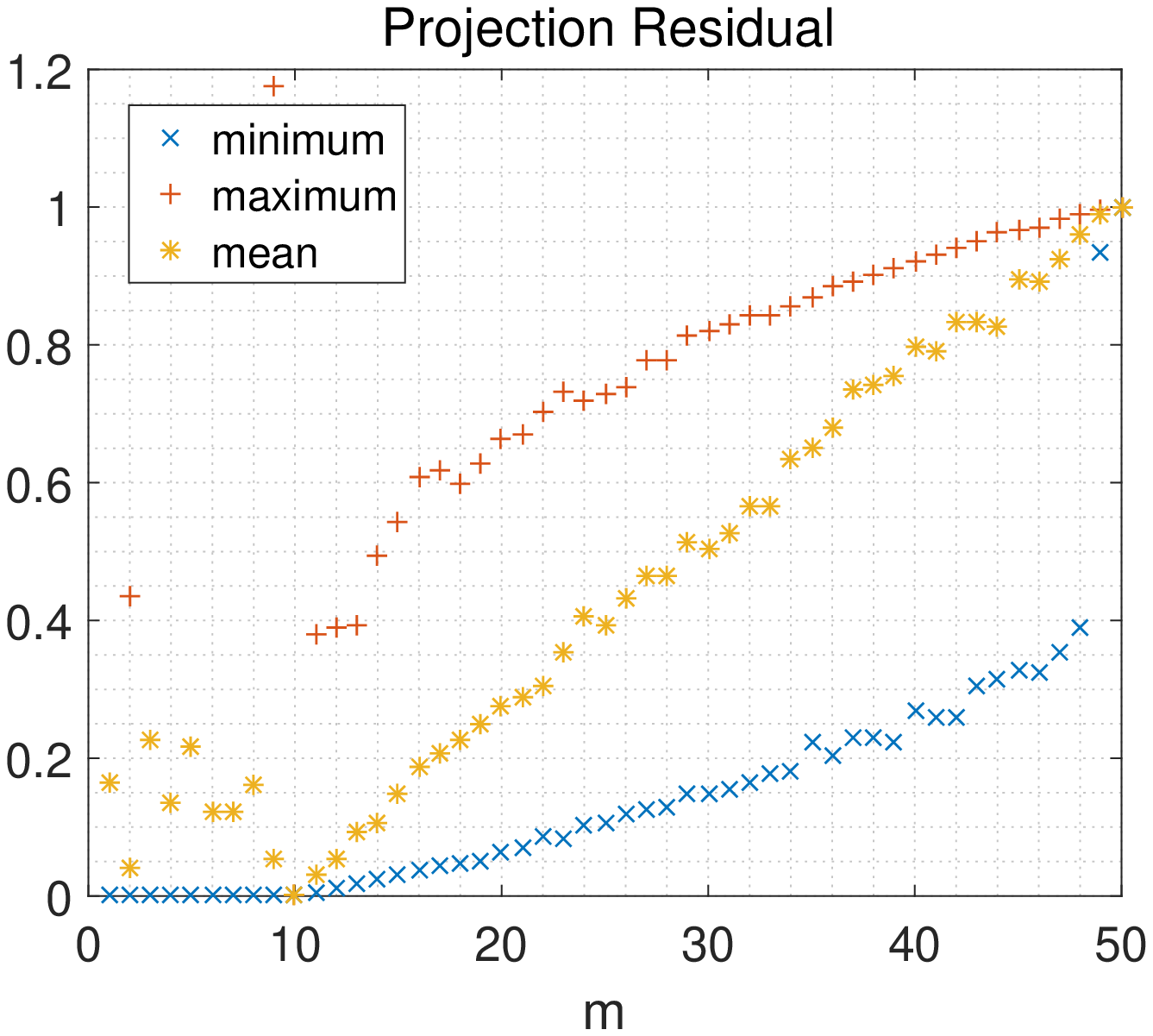}}
  \centerline{(a) $\mathcal{T}\in\mathcal{S}^\perp$}\medskip
\end{minipage}
\hfill
\begin{minipage}[t]{0.48\linewidth}
  \centering
  \centerline{\includegraphics[width=8cm]{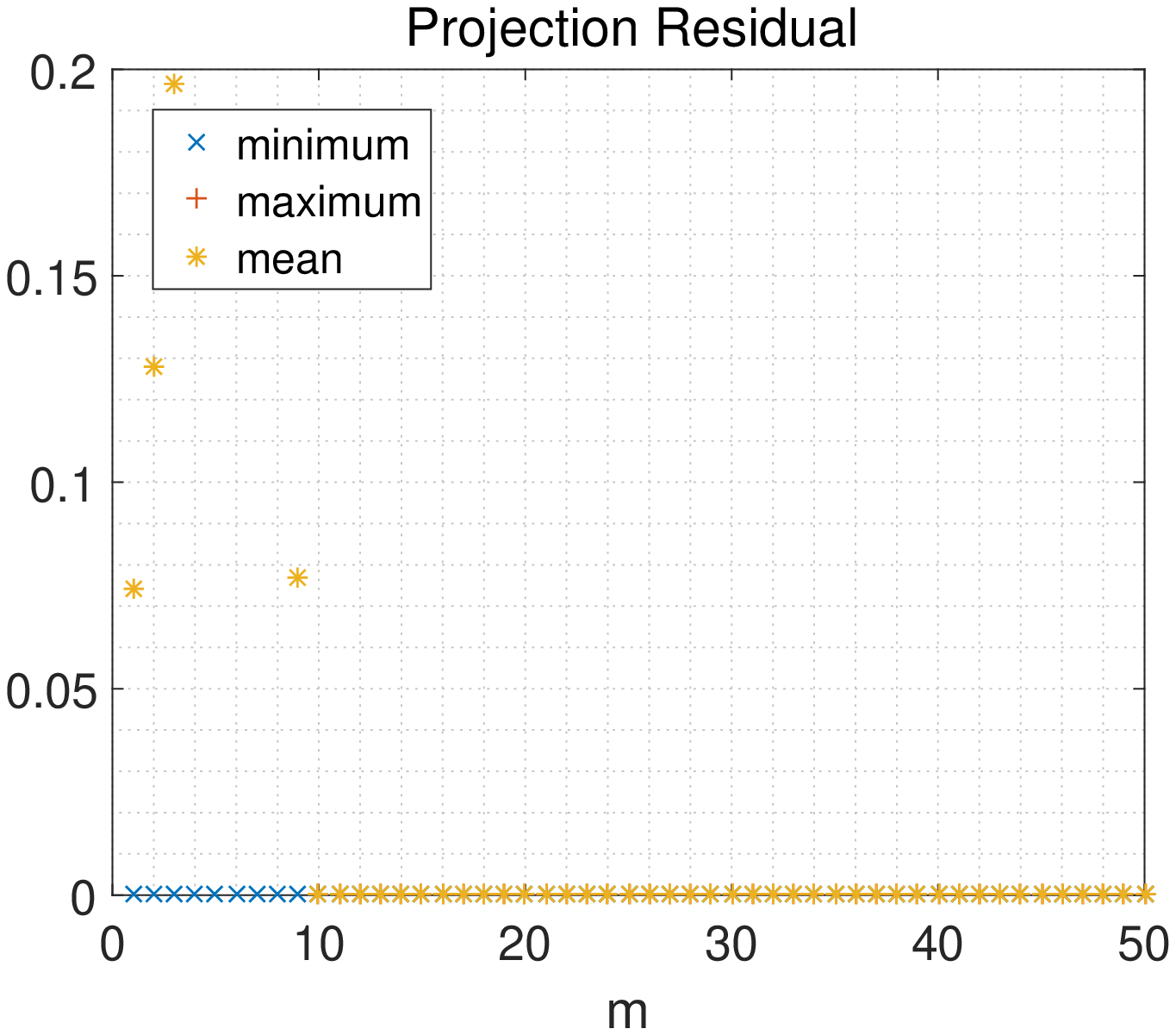}}
  \centerline{(b) $\mathcal{T}\in\mathcal{S}$}\medskip
\end{minipage}
\caption{Simulation results of $\|\mathcal{T}_\Omega-\mathcal{P}_\Omega\bullet\mathcal{T}_\Omega\|_F^2 $ based on DFT for tubal-sampling over 100 runs with $n_1=50 $, $n_2=10 $, $ n_3=50$ and $\mu(\mathcal{S})\approx 1.1 $. (a) $\mathcal{T}\in\mathcal{S} $, and (b) $\mathcal{T}\in\mathcal{S}^\bot $.}
\label{fig:tubal}
\end{figure}

\begin{figure}[t]

\begin{minipage}[t]{0.48\linewidth}
  \centering
  \centerline{\includegraphics[width=8cm]{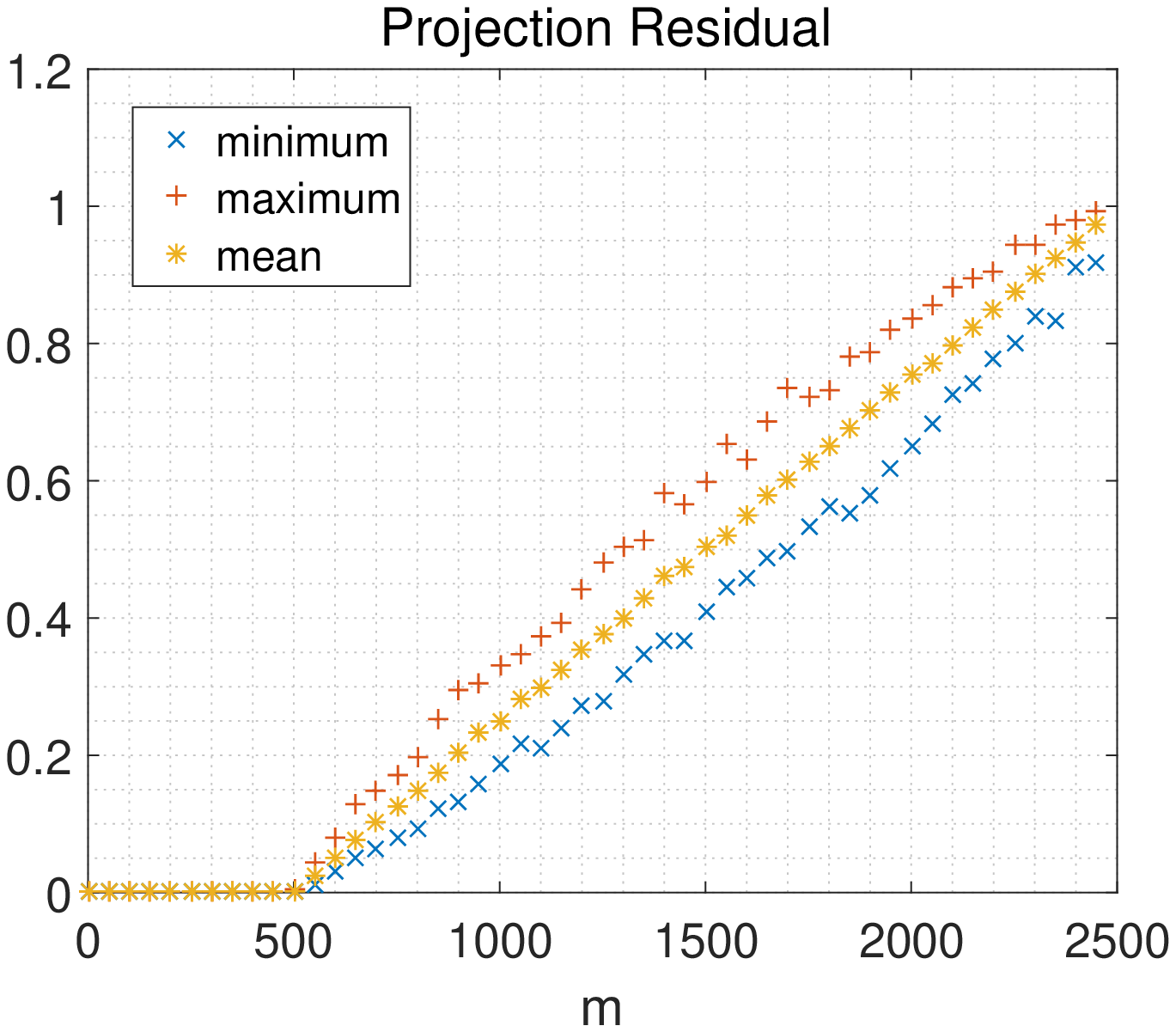}}
  \centerline{(a) $\mathcal{T}\in\mathcal{S}^\perp$}\medskip
\end{minipage}
\hfill
\begin{minipage}[t]{0.48\linewidth}
  \centering
  \centerline{\includegraphics[width=8cm]{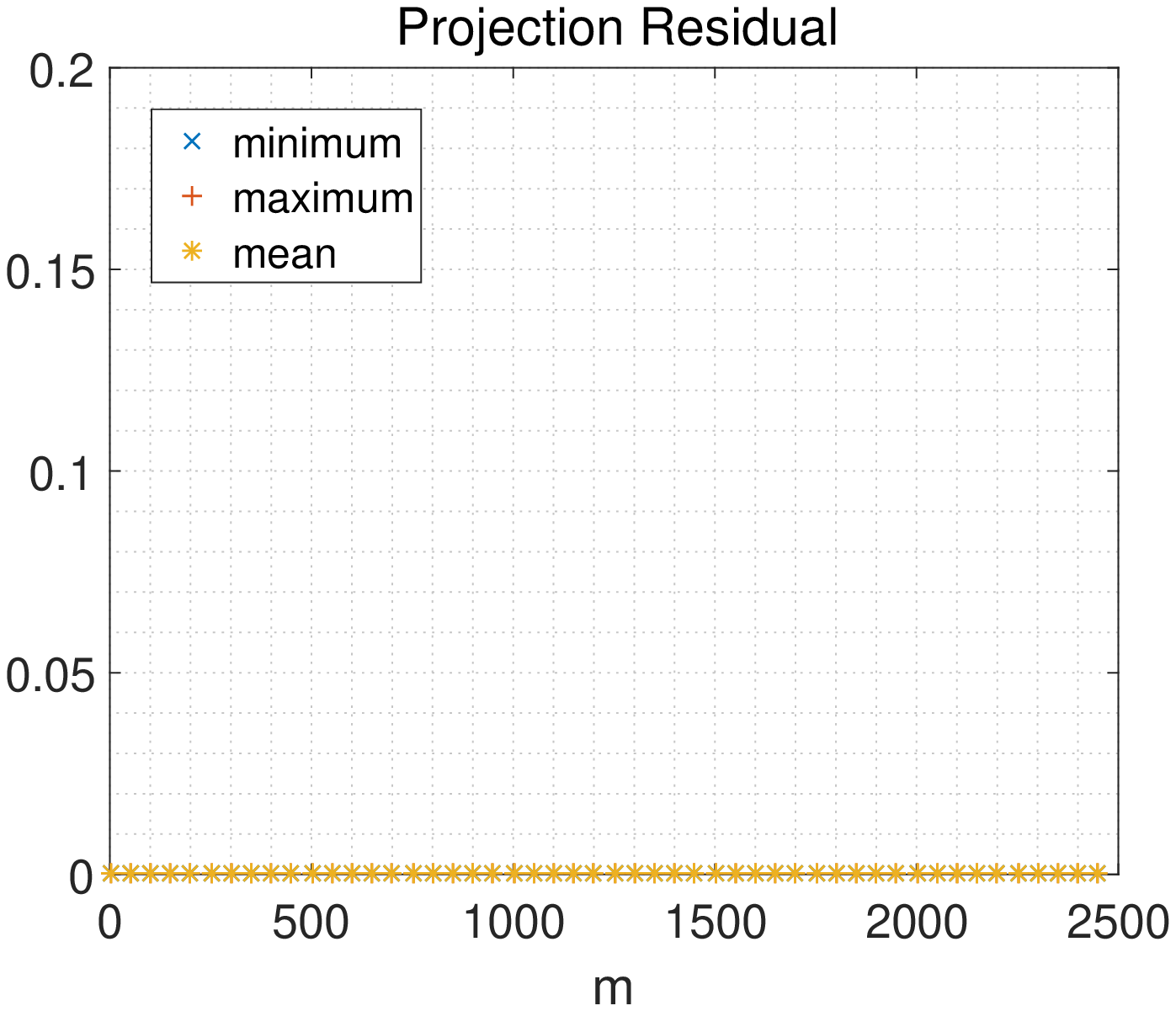}}
  \centerline{(b) $\mathcal{T}\in\mathcal{S}$}\medskip
\end{minipage}
\caption{Simulation results of $\left\|\bm{t}_{\Omega}-\bm{P}_{\Omega} \bm{t}_{\Omega}\right\|_2^2 $ based on DFT for elementwise-sampling over 100 runs with $n_1=50 $, $n_2=10 $, $ n_3=50$ and $\mu(\mathcal{S})\approx 1.1 $. (a) $\mathcal{T}\in\mathcal{S} $, and (b) $\mathcal{T}\in\mathcal{S}^\bot $.}
\label{fig:element}
\end{figure}

\begin{figure}[t]

\begin{minipage}[t]{0.48\linewidth}
  \center
  \centerline{\includegraphics[width=8cm]{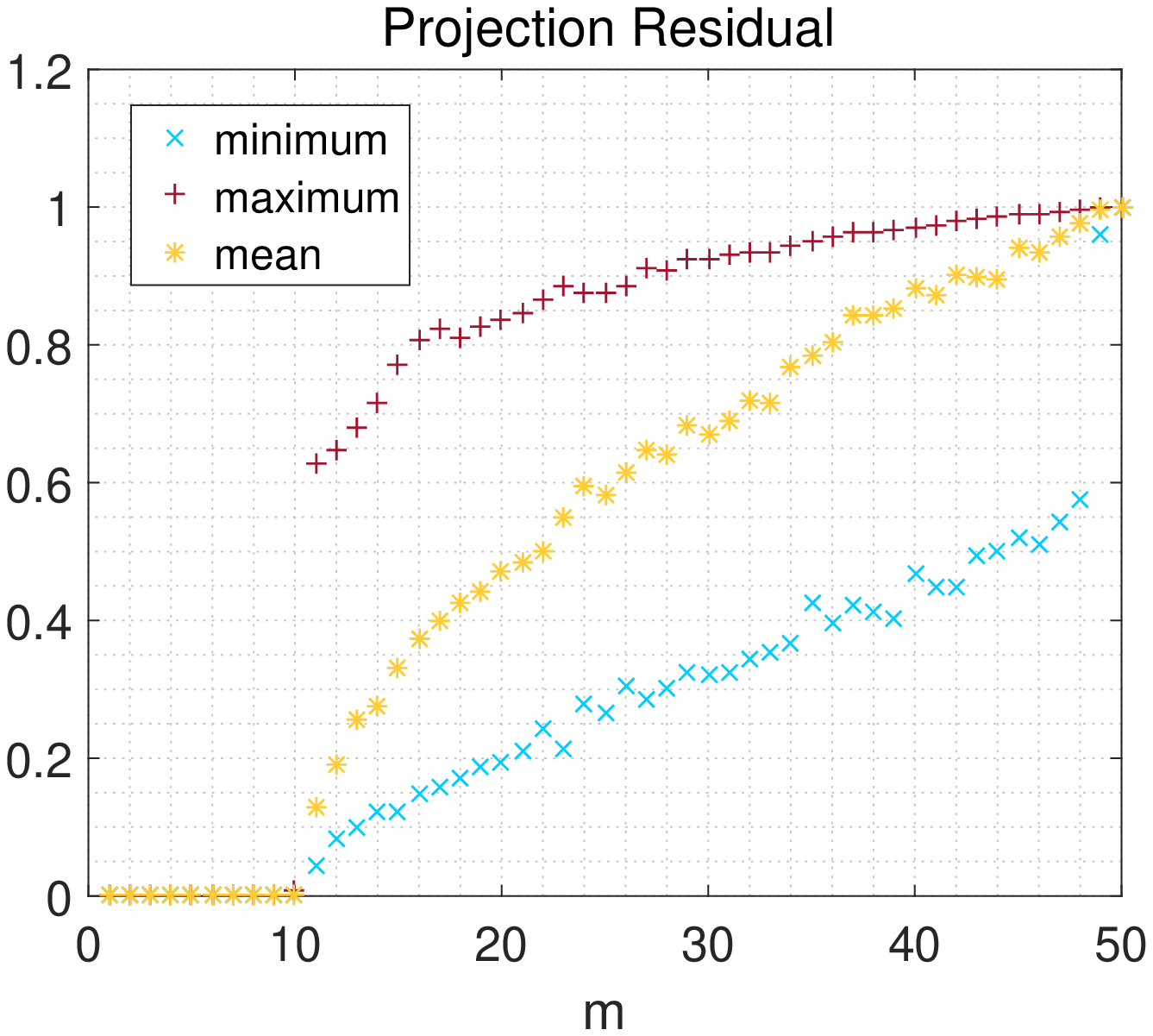}}
  \centerline{(a) $\mathcal{T}\in\mathcal{S}^\perp$}\medskip
\end{minipage}
\hfill
\begin{minipage}[t]{0.48\linewidth}
  \centering
  \centerline{\includegraphics[width=8cm]{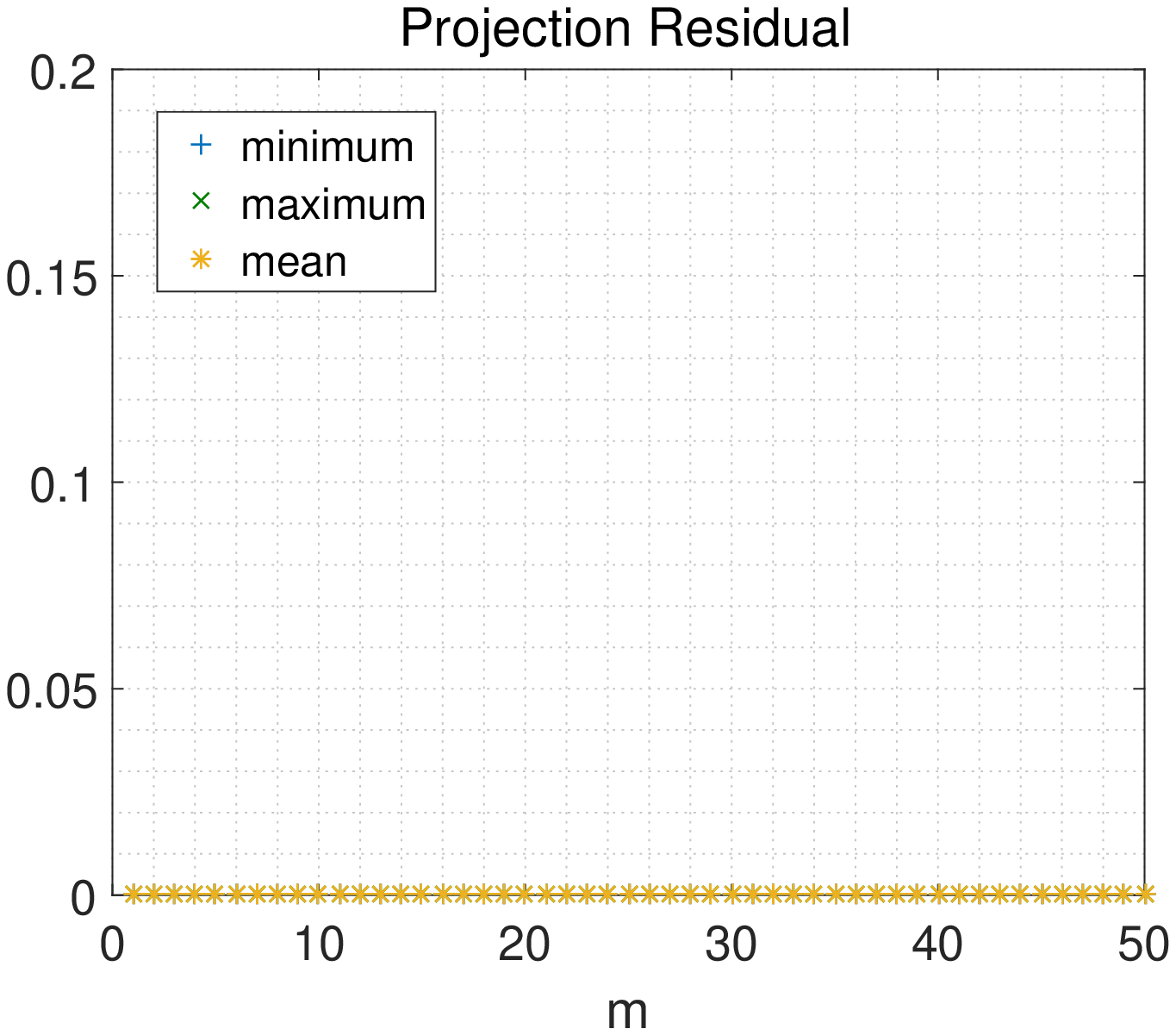}}
  \centerline{(b) $\mathcal{T}\in\mathcal{S}$}\medskip
\end{minipage}
\caption{Simulation results of $\|\mathcal{T}_\Omega-\mathcal{P}_\Omega\bullet\mathcal{T}_\Omega\|_F^2 $ based on DCT for tubal-sampling over 100 runs with $n_1=50 $, $n_2=10 $, $ n_3=50$ and $\mu(\mathcal{S})\approx 1.7 $. (a) $\mathcal{T}\in\mathcal{S} $, and (b) $\mathcal{T}\in\mathcal{S}^\bot $.}
\label{fig:dcttubal}
\end{figure}

\begin{figure}[t]

\begin{minipage}[t]{0.48\linewidth}
  \centering
  \centerline{\includegraphics[width=8cm]{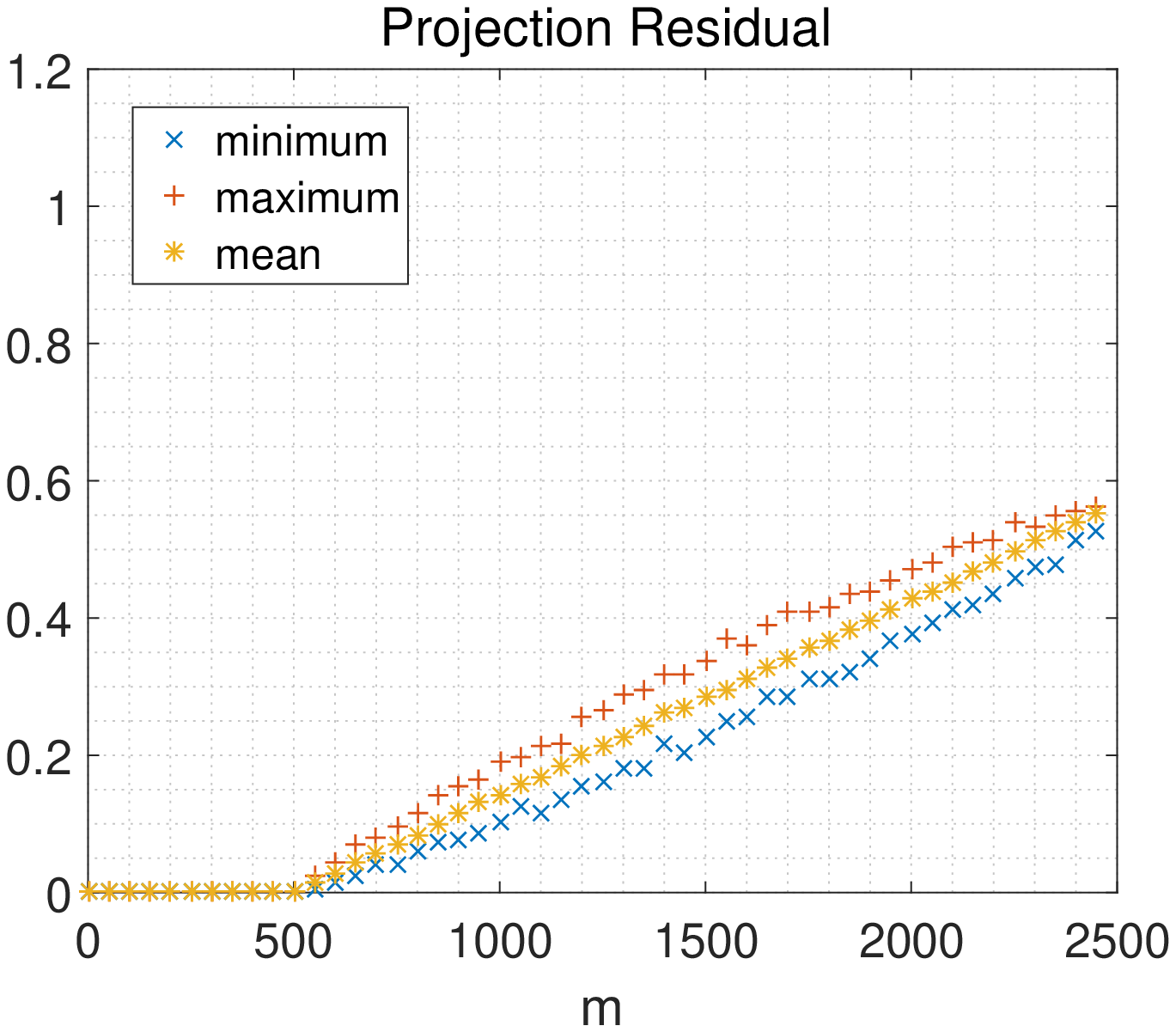}}
  \centerline{(a) $\mathcal{T}\in\mathcal{S}^\perp$}\medskip
\end{minipage}
\hfill
\begin{minipage}[t]{0.48\linewidth}
  \centering
  \centerline{\includegraphics[width=8cm]{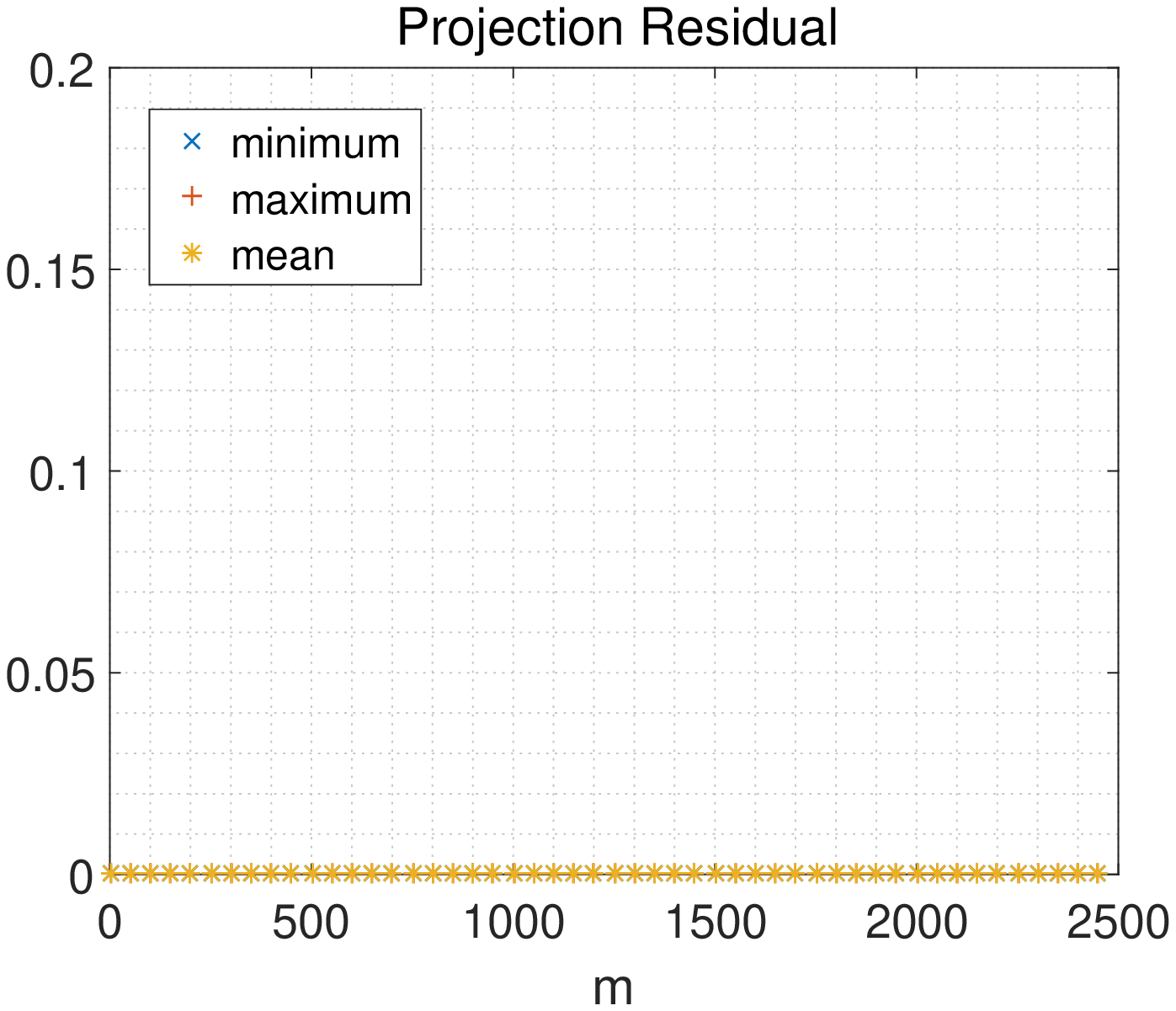}}
  \centerline{(b) $\mathcal{T}\in\mathcal{S}$}\medskip
\end{minipage}
\caption{Simulation results of $\left\|\bm{t}_{\Omega}-\bm{P}_{\Omega} \bm{t}_{\Omega}\right\|_2^2 $ based on DCT for elementwise-sampling over 100 runs with $n_1=50 $, $n_2=10 $, $ n_3=50$ and $\mu(\mathcal{S})\approx 1.7 $. (a) $\mathcal{T}\in\mathcal{S} $, and (b) $\mathcal{T}\in\mathcal{S}^\bot $.}
\label{fig:dctelement}
\end{figure}

Fig. \ref{fig:tubal}, and \ref{fig:element} show the performance of our estimators, and plots the minimum, maximum and mean value of $\left\|\mathcal{T}_{\Omega}-\mathcal{P}_{\Omega} \bullet\mathcal{T}_{\Omega}\right\|_F^2 $ and $\left\|\bm{t}_{\Omega}-\bm{P}_{\Omega} \bm{t}_{\Omega}\right\|_2^2 $ over $100$ simulations. For each value of the sample size $m$, we sample $100$ different $\Omega $ without replacement with fixed $\mathcal{T} $ and fixed $\mathcal{S} $.

Fig. \ref{fig:tubal} and Fig. \ref{fig:element} are based on DFT. Fig. \ref{fig:tubal} plots the projection residual energy $\left\|\mathcal{T}_{\Omega}-\mathcal{P}_{\Omega} \bullet\mathcal{T}_{\Omega}\right\|_F^2 $ with tubal-sampling. Fig. \ref{fig:tubal} (a) shows when $m$ is greater than the dimension of $\mathcal{S} $, $\left\|\mathcal{T}_{\Omega}-\mathcal{P}_{\Omega} \bullet\mathcal{T}_{\Omega}\right\|_F^2 >0$ for $\mathcal{T}\in\mathcal{S}^\bot $, and Fig. \ref{fig:tubal} (b) shows $\left\|\mathcal{T}_{\Omega}-\mathcal{P}_{\Omega} \bullet\mathcal{T}_{\Omega}\right\|_F^2 =0$ when $m>10$. However, when $m<10$, there exists some points that $\|\mathcal{T}_{\Omega}-\mathcal{P}_{\Omega} \bullet\mathcal{T}_{\Omega}\|_F^2>0 $ for $\mathcal{T}\in\mathcal{S} $, this is due to fast Fourier transform involves complex valued computation. Fig. \ref{fig:tubal} shows $\left\|\mathcal{T}_{\Omega}-\mathcal{P}_{\Omega} \bullet\mathcal{T}_{\Omega}\right\|_F^2 $ is approximate to $\frac{m}{n_1}\|\mathcal{T}-\mathcal{P}\bullet\mathcal{T}\|_F^2 $ with tubal-sampling. Fig. \ref{fig:element} plots the projection residual energy $\left\|\bm{t}_{\Omega}-\bm{P}_{\Omega} \bm{t}_{\Omega}\right\|_2^2 $ with elementwise-sampling. Similarly, Fig. \ref{fig:element} (a) and Fig. \ref{fig:element} (b) show $\left\|\bm{t}_{\Omega}-\bm{P}_{\Omega} \bm{t}_{\Omega}\right\|_2^2 $ is approximate to $\frac{m}{n_1n_3}\|\mathcal{T}-\mathcal{P}\bullet\mathcal{T}\|_F^2 $.

Fig. \ref{fig:dcttubal} and Fig. \ref{fig:dctelement} are based on DCT. Fig. \ref{fig:dcttubal} plots the projection residual energy $\left\|\mathcal{T}_{\Omega}-\mathcal{P}_{\Omega} \bullet\mathcal{T}_{\Omega}\right\|_F^2 $ with tubal-sampling, and Fig. \ref{fig:dctelement}  plots the projection residual energy $\left\|\bm{t}_{\Omega}-\bm{P}_{\Omega} \bm{t}_{\Omega}\right\|_2^2 $ with elementwise-sampling. Fig. \ref{fig:dcttubal} (a) and Fig. \ref{fig:dctelement} (a) show, for $\mathcal{T}\in\mathcal{S}^\bot $, the projection residual energy is always positive when $m>10$ with tubal-sampling and $m>10*50$ with elementwise-sampling. Fig. \ref{fig:dcttubal} (b) and Fig. \ref{fig:dctelement} (b) show the projection residual energy is always zero with any sample size for $\mathcal{T}\in\mathcal{S} $. Fig. \ref{fig:dcttubal} shows $\left\|\mathcal{T}_{\Omega}-\mathcal{P}_{\Omega} \bullet\mathcal{T}_{\Omega}\right\|_F^2 $ is approximate to $\frac{m}{n_1}\|\mathcal{T}-\mathcal{P}\bullet\mathcal{T}\|_F^2 $ for tubal-sampling, and Fig. \ref{fig:dctelement} shows $\left\|\bm{t}_{\Omega}-\bm{P}_{\Omega} \bm{t}_{\Omega}\right\|_2^2 $ is approximate to $0.58\frac{m}{n_1n_3}\|\mathcal{T}-\mathcal{P}\bullet\mathcal{T}\|_F^2 $.

\subsection{Performance Evaluation for Detectors}

\begin{figure}[t]

\begin{minipage}[t]{0.48\linewidth}
  \centering
  \centerline{\includegraphics[width=8cm]{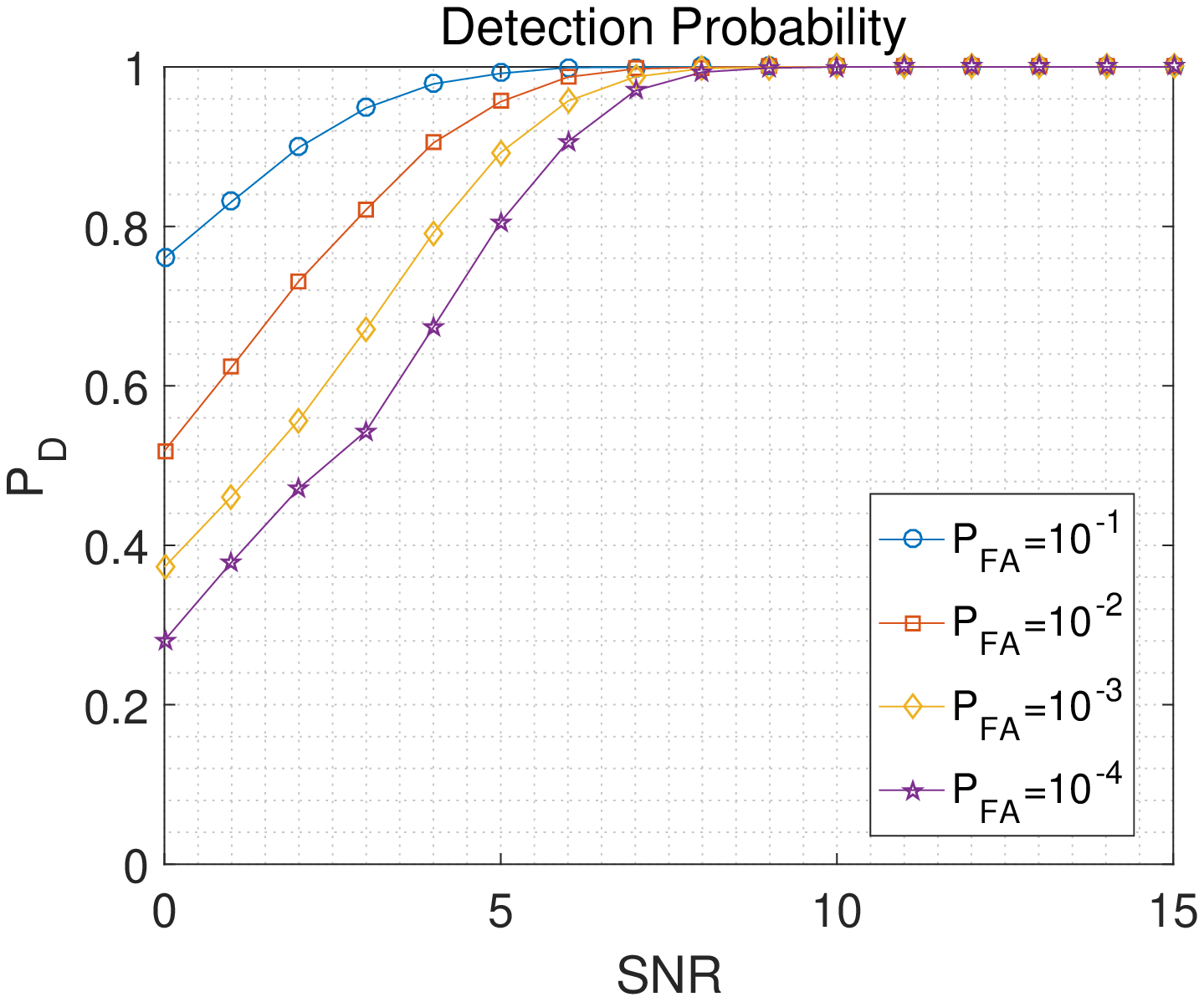}}
  \centerline{(a) tubal-sampling}\medskip
\end{minipage}
\hfill
\begin{minipage}[t]{0.48\linewidth}
  \centering
  \centerline{\includegraphics[width=8.4cm]{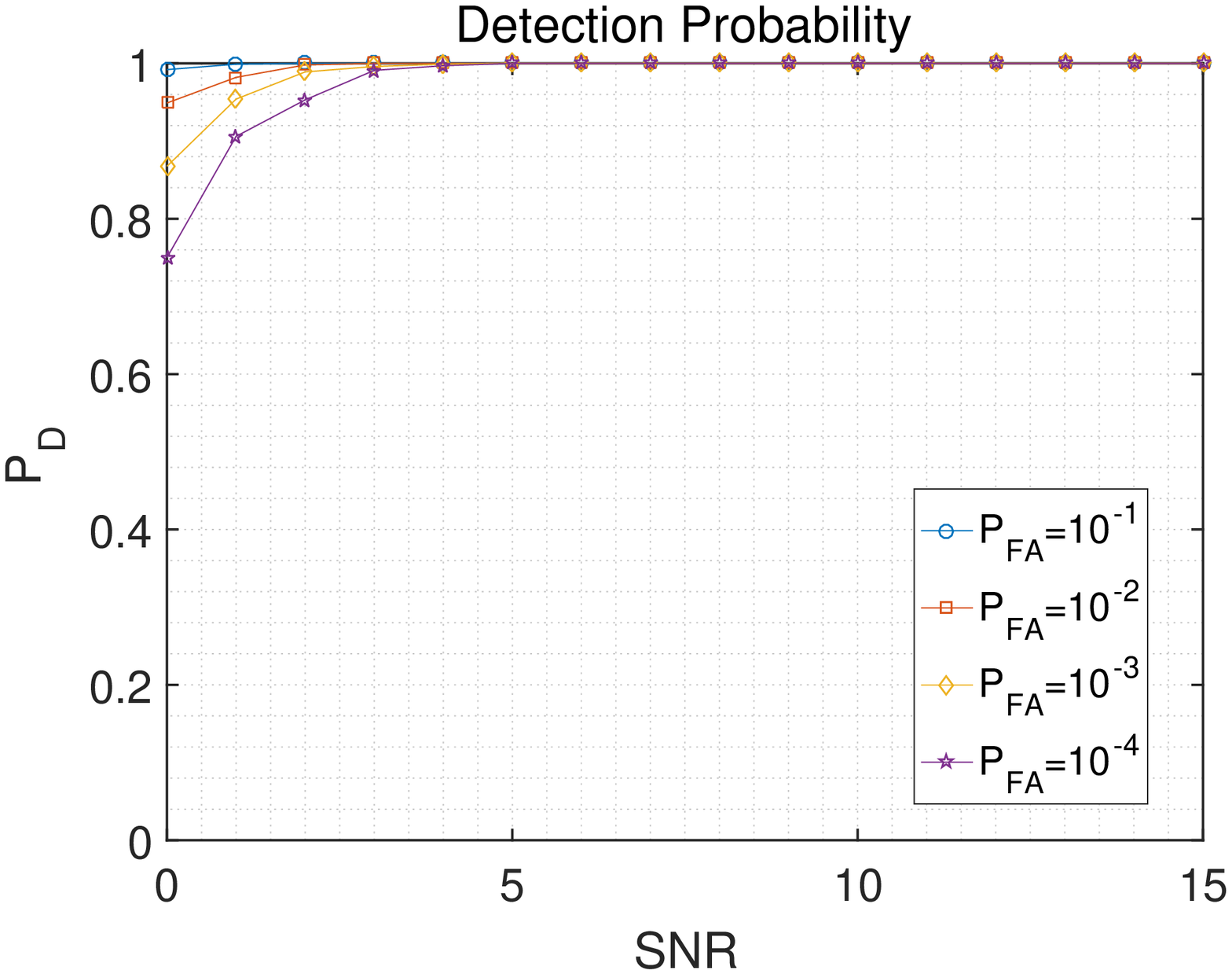}}
  \centerline{(b) elementwise-sampling}\medskip
\end{minipage}
\caption{The probability of detection based on DFT with different probabilities of false alarm. The sampling size $m=11$ for tubal-sampling, and $m=11\times 50$ for elementwise-sampling.}
\label{fig:detfft}
\end{figure}

\begin{figure}[t]

\begin{minipage}[t]{0.48\linewidth}
  \centering
  \centerline{\includegraphics[width=8cm]{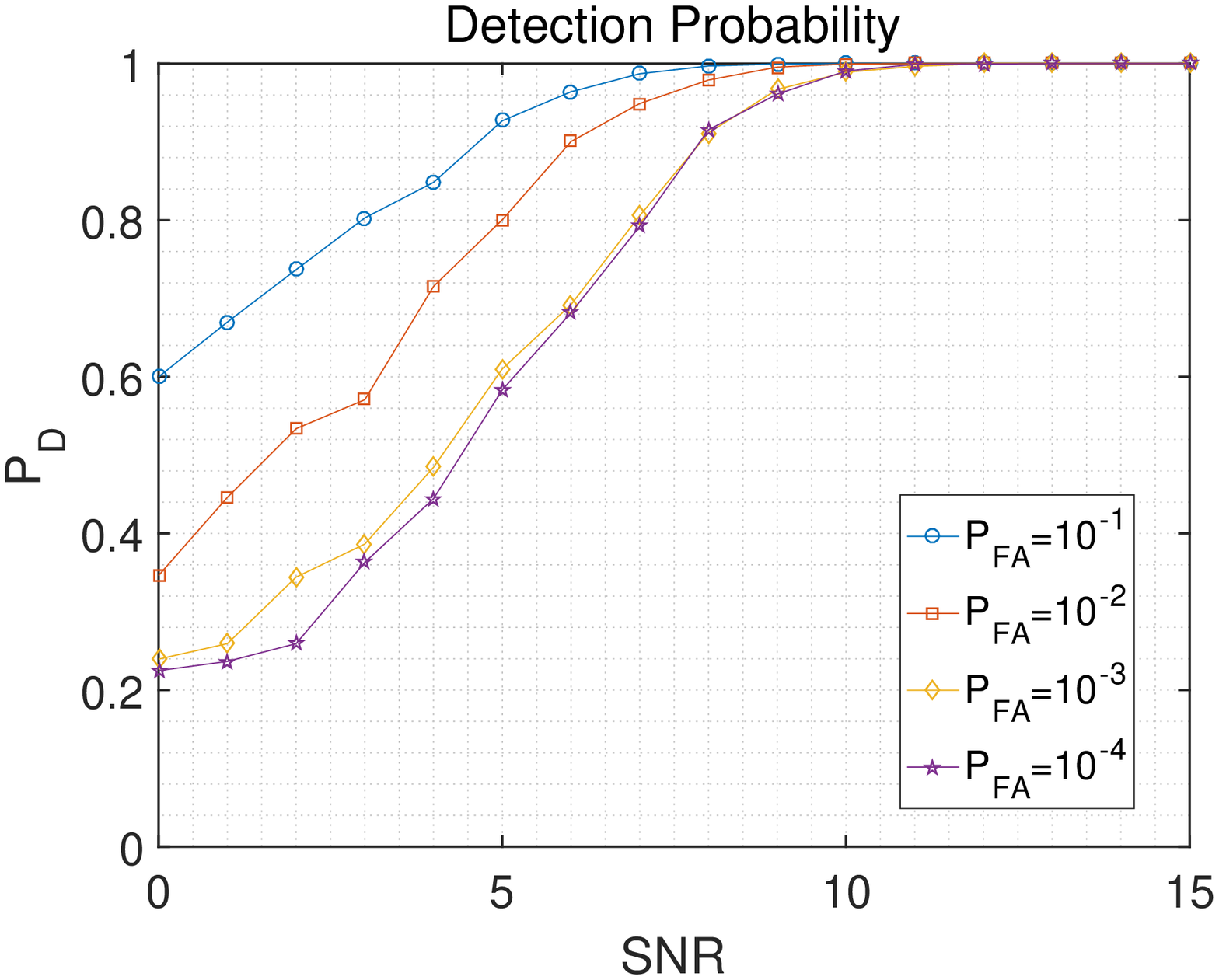}}
  \centerline{(a) tubal-sampling}\medskip
\end{minipage}
\hfill
\begin{minipage}[t]{0.48\linewidth}
  \centering
  \centerline{\includegraphics[width=8cm]{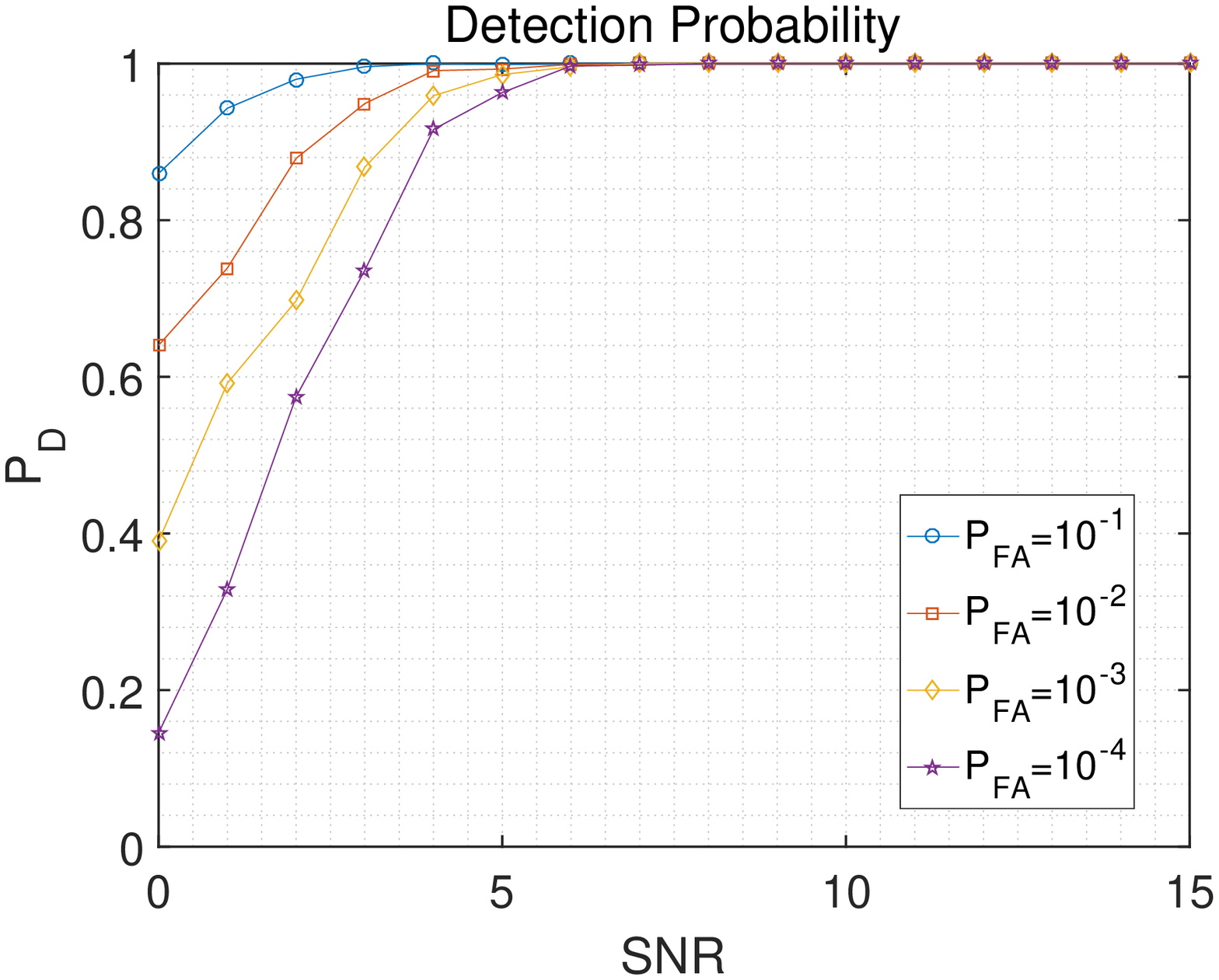}}
  \centerline{(b) elementwise-sampling}\medskip
\end{minipage}
\caption{The probability of detection based on DCT with different probabilities of false alarm. The sampling size $m=11$ for tubal-sampling, and $m=11\times 50$ for elementwise-sampling.}
\label{fig:detdct}
\end{figure}

\begin{figure}[t]

\begin{minipage}[b]{1\linewidth}
  \centering
  \centerline{\includegraphics[width=8cm]{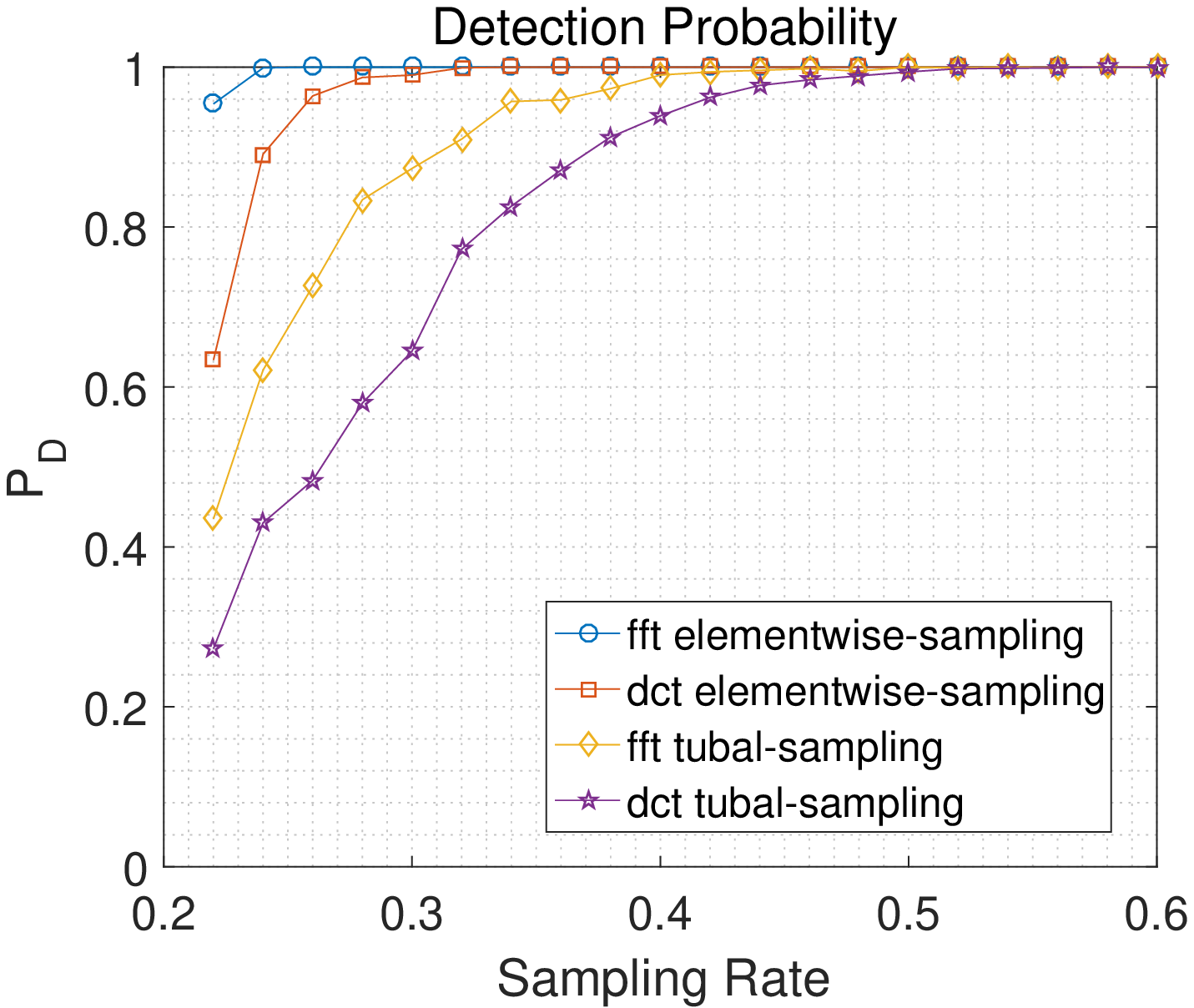}}
\end{minipage}
\caption{The probability of detection under different sampling rates with $\text{SNR}=0 $ and $P_{FA}=10^{-2} $.}
\label{fig:vs}
\end{figure}

Fig. \ref{fig:detfft}, \ref{fig:detdct}, and \ref{fig:vs} plots the detection probabilities under different conditions for fixed subspace $\mathcal{S}$ but different signal $\mathcal{T}$ based on DFT and DCT. With the sample sizes of $m=11$ for tubal-sampling and $m=11\times 50 $ for elementwise-sampling, Fig. \ref{fig:detfft} and \ref{fig:detdct} show the detection performance of our detector under different SNR and the different false alarm probabilities ($P_{FA} $), and Fig. \ref{fig:detfft} is based on DFT while Fig. \ref{fig:detdct} is based on DCT. From Fig. \ref{fig:detfft} and \ref{fig:detdct}, we find that the detection probability rises with the increase of the false alarm probability under the same SNR both for tubal-sampling and elementwise-sampling. Based on DFT,  the detection probability is approximate to $1$ when $\text{SNR}>8 $ for tubal-sampling and $\text{SNR}>5 $ for elementwise-sampling even $P_{FA}=10^{-4} $. Based on DCT,  the detection probability is approximate to $1$ when $\text{SNR}>10 $ for tubal-sampling and $\text{SNR}>6 $ for elementwise-sampling even $P_{FA}=10^{-4} $. Generally Speaking, under the same conditions, the performance of our detection based on DFT is superior to DCT, and elementwise-sampling is superior to tubal-sampling.

Fig. \ref{fig:vs} shows the comparison of the detection probability with $\text{SNR}= 0$ and $P_{FA}=10^{-2} $ under different sampling rate (the sampling rate is $\frac{m}{n_1}$ for tubal-sampling and $\frac{m}{n_1n_3}$ for elementwise-sampling) based on DFT and DCT. There are 4 detections: the detection with tubal-sampling based on DFT, the detection with elementwise-sampling based on DFT, the detection with tubal-sampling based on DCT, the detection with elementwise-sampling based on DCT. Among these detections, the detection with elementwise-sampling based on DFT has the best performance, of which the detection probability is approximate to $1$ as long as sampling rate $\frac{m}{n_1n_3}\geq 0.24 $.

\section{Conclusion}
\label{sec:conclusion}

 In this paper, we have proposed an approach for tensor matched subspace detection based on the transformed tensor model. Energy estimators under tubal-sampling and elementwise-sampling have been given, which estimator the energy of a signal outside a given subspace from sampling data. The bounds of energy estimators have been given, which have proved that it is possible  to detect whether a highly incomplete tensor belongs to a subspace when the number of samples is slightly greater than $r$ for tubal-sampling while $r\times n_3$ for elementwise-sampling. Matched subspace detections both for noiseless data and noisy data have been given. Moreover, simulations with synthetic data based on DFT and DCT have been given, which evaluate the performance of our estimators and detectors.

\bibliographystyle{IEEEbib}
\bibliography{tensor-subspace-detection}

\begin{thebibliography}{10}

\bibitem{AKAS2013}
A.~Krishnamurthy and A.~Singh,
\newblock ``Low-rank matrix and tensor completion via adaptive sampling,''
\newblock {\em Neural Information Processing Systems (NIPS)}, 2013.

\bibitem{DBLP2011}
Brian Eriksson, Laura Balzano, and Robert~D. Nowak,
\newblock ``High-rank matrix completion and subspace clustering with missing
  data,''
\newblock {\em CoRR}, vol. abs/1112.5629, 2011.

\bibitem{Pimentel2015}
D.~L. Pimentel-Alarc¨®n, N.~Boston, and R.~D. Nowak,
\newblock ``Deterministic conditions for subspace identifiability from
  incomplete sampling,''
\newblock in {\em 2015 IEEE International Symposium on Information Theory
  (ISIT)}, June 2015, pp. 2191--2195.

\bibitem{SD2017}
Jun Wang and Kai Xu,
\newblock ``Shape detection from raw lidar data with subspace modeling,''
\newblock {\em IEEE Transactions on Visualization and Computer Graphics}, vol.
  23, no. 9, pp. 2137--2150, August 2017.

\bibitem{SubRe2014}
Yeqing Li, Chen Chen, and Junzhou Huang,
\newblock ``Transformation-invariant collaborative sub-representation,''
\newblock in {\em IEEE 22nd International Conference on Pattern Recognition}.
  IEEE, 2014, pp. 3738--3743.

\bibitem{MIMO2016}
Hadi Sariedeen, Mohammad~M. Mansour, and Ali Chehab,
\newblock ``Efficient subspace detection for high-order mimo systems,''
\newblock in {\em 2016 IEEE International Conference on Acoustics, Speech and
  Signal Processing (ICASSP)}. IEEE, 2016, pp. 1001--1005.

\bibitem{MIMO2017}
Mohammad~M. Mansour,
\newblock ``A near-ml mimo subspace detection algorithm,''
\newblock {\em IEEE Signal Processing Letters}, vol. 22, pp. 408--412, April
  2015.

\bibitem{rf2016}
Xiao-Yang Liu, Shuchin Aeron, Vaneet Aggarwal, Xiaodong Wang, and Min-You Wu,
\newblock ``Adaptive sampling of rf fingerprints for fine-grained indoor
  localization,''
\newblock {\em IEEE Transactions on Mobile Computing}, vol. 15, no. 10, pp.
  2411--2423, October 2016.

\bibitem{XYL2016ICASSP}
Xiao-Yang Liu, Shuchin Aeron, Vaneet Aggarwal, and Xiaodong Wang,
\newblock ``Tensor completion via adaptive sampling of tensor
  fibers:application to efficient indoor rf fingerprinting,''
\newblock in {\em 2016 IEEE International Conference on Acoustics, Speech and
  Signal Processing (ICASSP)}. IEEE, 2016, pp. 2529--2533.

\bibitem{LB1994}
L.~Scharf and B.~Friedlander,
\newblock ``Matched subspace detectors,''
\newblock {\em IEEE Transactions on Signal Processing}, vol. 42, pp.
  2146--2157, August 1994.

\bibitem{Paredes2009}
J.~L. Paredes, Z.~Wang, G.~R. Arce, and B.~M. Sadler,
\newblock ``Compressive matched subspace detection,''
\newblock in {\em 2009 17th European Signal Processing Conference}, Aug 2009,
  pp. 120--124.

\bibitem{Mcwhorter2001Matched}
Todd Mcwhorter and Louis~L Scharf,
\newblock ``Matched subspace detectors for stochastic signals,''
\newblock {\em Signal Processing IEEE Transactions on}, vol. 42, no. 8, pp.
  2146 -- 2157, 2001.

\bibitem{IOPORT01742117}
Louis~L. Scharf and Shawn Kraut,
\newblock ``Geometries, invariances, and snr interpretations of matched and
  adaptive subspace detectors.,''
\newblock {\em Traitement du Signal}, vol. 15, no. 6, pp. 527--534, 1998.

\bibitem{8068230}
X.~Y. Liu and X.~Wang,
\newblock ``Ls-decomposition for robust recovery of sensory big data,''
\newblock {\em IEEE Transactions on Big Data}, vol. PP, no. 99, pp. 1--1, 2017.

\bibitem{Balzano2010}
Laura Balzano, Benjamin Recent, and Robert Nowak,
\newblock ``High-dimentional matched subspace detection when data are
  missing,''
\newblock {\em IEEE International Symposium on Information Theory}, pp.
  1638--1642, June 2010.

\bibitem{Balzano2017}
Dejiao Zhang and Laura Balzano,
\newblock ``Matched subspace detection using compressively sampled data,''
\newblock in {\em 2017 IEEE International Conference on Acoustics, Speech and
  Signal Processing (ICASSP)}. IEEE, 2017, pp. 4601--4605.

\bibitem{kernel2014}
Tong Wu and Waheed~U. Bajwa,
\newblock ``Subspace detection in a kernel space: The missing data case,''
\newblock in {\em 2014 IEEE Workshop on Statistical Processing (SSP)}. IEEE,
  2014, pp. 93--96.

\bibitem{Martin2012}
Martin Azizyan and Aarti Singh,
\newblock ``Subspace detection of high-dimensional vectors using compressive
  sampling,''
\newblock in {\em 2012 IEEE Workshop on Statistical Processing (SSP)}. IEEE,
  2012, pp. 724--727.

\bibitem{Baraniuk2009}
Tmamara~G. Kolda and B.W. Bader,
\newblock ``Tensor decompositions and applications,''
\newblock {\em SIAM Review}, vol. 51, pp. 455--500, August 2009.

\bibitem{Cichocki2015}
A.~Cichocki, D.~Mandic, L.~De Lathauwer, G.~Zhou, Q.~Zhao, C.~Caiafa, and H.~A.
  PHAN,
\newblock ``Tensor decompositions for signal processing applications: From
  two-way to multiway component analysis,''
\newblock {\em IEEE Signal Processing Magazine}, vol. 32, no. 2, pp. 145--163,
  March 2015.

\bibitem{Zhang2017}
Zemin Zhang and Shuchin Aeron,
\newblock ``Exact tensor completion using t-svd,''
\newblock {\em IEEE Transactions on Signal Processing}, vol. 65, no. 6, pp.
  1511--1526, March 2017.

\bibitem{XYL2017TSP}
Xiao-Yang Liu and Xiaodong Wang,
\newblock ``Fourth-order tensors with multidimensional discrete transforms,''
\newblock in {\em arXiv preprint arXiv:1705.01576}, 2017, pp. 1--37.

\bibitem{TTPRO2015}
E.~Kernfeld, M.~Kilmer, and S.~Aeron,
\newblock ``Tensor-tensor products with invertible linear transforms,''
\newblock {\em Linear Algebra and its Applications}, vol. 485, pp. 545--570,
  November 2015.

\bibitem{Kilmer2013}
M.E. Kilmer, K.~Braman, N.~Hao, and R.C. Hoover,
\newblock ``Third-order tensors as operators on matrices: A theoretical and
  computational framework with applications in imaging,''
\newblock {\em SIAM Journal on Matrix Analysis and Applications}, vol. 34, pp.
  148--172, February 2013.

\bibitem{DCT1995}
V.~Sanchez, P.~Garcia, A.~M. Peinado, J.~C. Segura, and A.~J. Rubio,
\newblock ``Diagonalized properties of the discrete cosine transforms,''
\newblock {\em IEEE Transactions on Signal Processing}, vol. 43, no. 11, pp.
  2631--2641, 1995.

\bibitem{TK2005}
T.~Kailath and Vadim Olshevsky,
\newblock ``Displacement structure approach to discrete-trigonometric-transform
  based precondetioners of g. strang type and of t. chan type,''
\newblock {\em SIAM Journal on Matrix Analysis and Applications}, vol. 26, no.
  3, pp. 706--734, 2005.

\bibitem{LIU2009853}
Huan Liu, Yongqiang Tang, and Hao~Helen Zhang,
\newblock ``A new chi-square approximation to the distribution of non-negative
  definite quadratic forms in non-central normal variables,''
\newblock {\em Computational Statistics and Data Analysis}, vol. 53, no. 4, pp.
  853 -- 856, 2009.

\end{thebibliography}

\appendix[proof of lemmas]
The following three versions of Bernstein's inequalities are needed in the proofs of our Lemmas.

\begin{lemma}\label{lem:scar} [Scalar Version \cite{rf2016}] Let $x_1,\ldots,x_m$ be independent zero-mean scalar variables. Suppose $\rho_k=\mathbb{E}[x_k^2]$ and $|x_k|\leq M$ almost surely for all $k$. Then for any $\tau>0$,
\begin{equation}
  \mathbb{P}\left[\sum\limits_{k=1}^mx_k\geq\tau\right]\leq \exp\left(\frac{-\tau^2/2}{\sum\limits_{k=1}^m\rho_k^2+M\tau/3}\right).
\end{equation}
\end{lemma}

\begin{lemma}\label{lem:vec} [Vector Version \cite{rf2016}] Let $\bm{x}_1,\ldots,\bm{x}_m$ be independent zero-mean random vectors with $\sum_{k=1}^m\mathbb{E}\|\bm{x}_k\|_2^2\leq M $. Then for any $\tau\leq M\left(max_i\|\bm{x}_i\|_2\right)^{-1} $,
\begin{equation}
  \mathbb{P}\left[\left\|\sum\limits_{k=1}^m\bm{x}_k\right\|_2\geq\sqrt{ M}+\tau\right]\leq \exp\left(\frac{-\tau^2}{4 M}\right).
\end{equation}
\end{lemma}

\begin{lemma}\label{lem:mat} [Matrix Version \cite{Balzano2010}] Let $\bm{X}_1,\ldots,\bm{X}_m$ be independent zero-mean square $r\times r$ random matrices. Suppose $\rho_k=\max\{\|\mathbb{E}[\bm{X}_k\bm{X}_k^{H}]\|,\|\mathbb{E}[\bm{X}_k^{H} \bm{X}_k]\|\}$ and $\|\bm{X}_k\|\leq M$ almost surely for all $k$. Then for any $\tau>0$,
\begin{equation}
  \mathbb{P}\left[\left\|\sum\limits_{k=1}^m\bm{X}_k\right\|>\tau\right]\leq 2r\exp\left(\frac{-\tau^2/2}{\sum\limits_{k=1}^m\rho_k^2+M\tau/3}\right).
\end{equation}
\end{lemma}

Based on the above Bernstein's inequalities, we can gain the proofs of our six central Lemmas.

\begin{proof}[Proof of Lemma \ref{tubal:lemma1}]
We use Bernstein's inequality of Lemma \ref{lem:scar} to prove Lemma \ref{tubal:lemma1}. Recall that $\Omega\subset [n_1] $ for tubal-sampling, and we assume the tubal samples are taken uniformly with replacement. For $\mathbb{E}\left[\left\|\mathcal{Y}_{\Omega}(i,:,:)\right\|_2^2\right]= \frac{1}{n_1}\sum\limits_{j=1}^{n_1}\left\|\mathcal{Y}(j,:,:)\right\|_2^2 =\frac{1}{n_1}\left\|\mathcal{Y}\right\|_F^2 $, we set $ x_i=\left\|\mathcal{Y}_{\Omega}(i,:,:)\right\|_2^2- \frac{1}{n_1}\left\|\mathcal{Y}\right\|_F^2$, such that $\mathbb{E}\left[x_i\right]=0$. Define $\bm{1}_\Omega $ as the indicator function, and we have
\begin{eqnarray*}
  \mathbb{E}\left[x_i^2\right] &=& \mathbb{E}\left[\left(\left\|\mathcal{Y}_{\Omega}(i,:,:)\right\|_2^2- \frac{1}{n_1}\left\|\mathcal{Y}\right\|_F^2\right)^2\right] \\
   &=&\mathbb{E}\left[\sum\limits_{j=1}^{n_1}\left\|\mathcal{Y}(j,:,:)\right\|_2^4 \bm{1}_{\Omega(i)=j}\right] -\frac{1}{n_1^2}\left\|\mathcal{Y}\right\|_F^4 \\
   &=&\frac{1}{n_1}\sum\limits_{j=1}^{n_1}\left\|\mathcal{Y}(j,:,:)\right\|_2^4 -\frac{1}{n_1^2}\left\|\mathcal{Y}\right\|_F^4  \\
   &\leq&\frac{1}{n_1}\left\|\mathcal{Y}\right\|_{\infty^*}^2\left\|\mathcal{Y}\right\|_F^2 -\frac{1}{n_1^2}\left\|\mathcal{Y}\right\|_F^4
\end{eqnarray*}
and
\begin{equation*}
  |x_i|\leq\left\|\mathcal{Y}\right\|_{\infty^*}^2-\frac{1}{n_1}\left\|\mathcal{Y}\right\|_F^2,
\end{equation*}
we set $\rho_i^2=\frac{1}{n_1}\left\|\mathcal{Y}\right\|_{\infty^*}^2\left\|\mathcal{Y}\right\|_F^2 -(\frac{1}{n_1})^2\left\|\mathcal{Y}\right\|_F^4 $ and $M=\left\|\mathcal{Y}\right\|_{\infty^*}^2-\frac{1}{n_1}\left\|\mathcal{Y}\right\|_F^2 $. Now we apply Bernstein's inequality of Lemma \ref{lem:scar}:
\begin{equation*}
   \mathbb{P}\left[\left|\sum\limits_{i=1}^mx_{i} \right|>\tau\right]\leq 2\exp\left(\frac{-\tau^2/2}{\sum_{i=1}^m\rho_{i}^2+M\tau/3}\right)
 \end{equation*}
 Let $\tau=\alpha\frac{m}{n_1}\left\|\mathcal{Y}\right\|_F^2 $, where $\alpha $ is defined in Theorem \ref{theo:tubal}, then we have
 \begin{equation*}
   (1-\alpha)\frac{m}{n_1}\|\mathcal{Y}\|_F^2\leq \left\|\mathcal{Y}_{\Omega}\right\|_F^2\leq (1+\alpha)\frac{m}{n_1}\|\mathcal{Y}\|_F^2
 \end{equation*} with the probability at least $1-2\delta $.
\end{proof}

\begin{proof}[Proof of lemma \ref{tubal:lemma2}]
Let $\bm{x}_i=\mathcal{U}^\dag(:,\Omega(i),:)\bullet\mathcal{Y}(\Omega(i),1,:) $. Since $\mathcal{Y}\in \mathcal{S}^\bot$ and the samples are taken uniformly with replacement, we have the followings.
\begin{eqnarray*}
  \mathbb{E}\left[\bm{x}_i\right] &=& \mathbb{E}\left[\sum\limits_{j=1}^{n_1}\mathcal{U}^\dag(:,j,:)\bullet\mathcal{Y}(j,1,:)\bm{1}_{\Omega(i)=j}\right] \\
   &=& \frac{1}{n_1}\sum\limits_{j=1}^{n_1}\mathcal{U}^\dag(:,j,:)\bullet\mathcal{Y}(j,1,:)=\bm{0}_{n_2\times 1\times n_3},
\end{eqnarray*}
\begin{eqnarray*}
  \sum\limits_{i=1}^{m}\mathbb{E}\left\|\bm{x}_i\right\|_2^2 &=& \sum\limits_{i=1}^{m}\mathbb{E}\left[\sum\limits_{j=1}^{n_1}\mathcal{U}^\dag(:,j,:)\bullet\mathcal{Y}(j,1,:)\bm{1}_{\Omega(i)=j}\right] \\
  &=& \frac{m}{n_1}\sum\limits_{j=1}^{n_1}\left\|\mathcal{U}^\dag(:,j,:)\bullet\mathcal{Y}(j,1,:)\right\|_F^2 \\
  &=&\frac{m}{n_1}\sum\limits_{j=1}^{n_1} \frac{1}{c^2}\left\|\overline{\mathcal{U}(:,j,:)}^T\overline{\mathcal{Y}(j,1,:)}\right\|_F^2   \\
   &\leq& \frac{m}{n_1}\frac{n_2\mu(\mathcal{S})}{n_1}c^2\left\|\mathcal{Y}\right\|_F^2\\
   &=&\frac{mn_2c^2\mu(\mathcal{S})}{n_1^2}\left\|\mathcal{Y}\right\|_F^2.
\end{eqnarray*}
We set $ M=\frac{mn_2c^2\mu(\mathcal{S})}{n_1^2}\left\|\mathcal{Y}\right\|_F^2$ and $\tau=\sqrt{4 M\ln(1/\delta)} $. When $m\geq 4\frac{n_1\|\mathcal{Y}\|_{\infty^*}^2}{\|\mathcal{Y}\|_F^2}\ln\left(\frac{1}{\delta}\right) $, $\sqrt{4 M\ln(1/\delta)}\leq M\left(\max\limits_i\|\bm{x}_i\|_2\right)^{-1} $. Applying Bernstein's inequality of Lemma \ref{lem:vec}, we have
\begin{eqnarray*}
  \left\|\mathcal{U}_{\Omega}^\dag\bullet\mathcal{Y}_{\Omega}\right\|_F^2 &\leq& (\sqrt{ M}+\sqrt{4 M\ln(1/\delta)})^2 \\
   &=& \beta\frac{mn_2c^2\mu(\mathcal{S})}{n_1^2}\left\|\mathcal{Y}\right\|_F^2
\end{eqnarray*}
with the probability at least $1-\delta $, where $\beta $ is defined in Theorem \ref{theo:tubal}.
\end{proof}

\begin{proof}[Proof of Lemma \ref{tubal:lemma3}]
We apply Bernstein's inequality of Lemma \ref{lem:mat} to prove Lemma \ref{tubal:lemma3}. Let $ \bm{X}_k=\overline{\mathcal{U}_{{\Omega}(k)}}^{H}\overline{\mathcal{U}_{{\Omega}(k)}} -\frac{1}{n_1}\bm{I}_{n_2n_3}$, where the notation $\mathcal{U}_{{\Omega}(k)} $ is the ${\Omega}(k) $th row of $\mathcal{U} $ and $\bm{I}_{n_2n_3}$ is the identity matrix of size $n_2n_3\times n_2n_3 $. Then the random variable $\bm{X}_k$ is zero mean. For ease of notation, we will denote $\overline{\mathcal{U}_{{\Omega}(k)}}$ as $\overline{\mathcal{U}_k}$.
Using the fact that for positive semi-defined matrices $\bm{A}$ and $ \bm{B}$, $\|\bm{A}-\bm{B}\|\leq \max\{\|\bm{A}\|,~\|\bm{B}\|\} $, and  recalling that $\|\overline{\mathcal{U}_k}\|_F^2= \left\|\mathcal{L}(\mathcal{U})(\Omega(k),:,:)\right\|_F^2 =\left\|\mathcal{L}(\mathcal{P})(:,\Omega(k),:)\right\|_F^2\leq c^2n_2\mu(\mathcal{S})/n_1$, we have
\begin{equation*}
  \left\|\overline{\mathcal{U}_{k}}^{H}\overline{\mathcal{U}_{k}} -\frac{1}{n_1}\bm{I}_{n_2n_3}\right\|
\leq\max\left\{\frac{c^2n_2\mu(\mathcal{S})}{n_1},\frac{1}{n_1}\right\}.
\end{equation*}
Let $M=\frac{c^2n_2\mu(\mathcal{S})}{n_1}$.
Next, we calculate $\|\mathbb{E}[\bm{X}_k\bm{X}_k^{H} ]\|_2$ and $\|\mathbb{E}[\bm{X}_k^{H}\bm{X}_k]\|_2$.
\begin{eqnarray*}
  \left\|\mathbb{E}\left[\bm{X}_k\bm{X}_k^{H} \right]\right\| &=& \left\|\mathbb{E}\left[\bm{X}_k^{H} \bm{X}_k\right]\right\| \\
   &=& \left\|\mathbb{E}\left[\left(\overline{\mathcal{U}_{k}}^{H}\overline{\mathcal{U}_{k}} -\frac{1}{n_1}\bm{I}_{n_2n_3}\right)^2\right]\right\| \\
   &=& \left\|\mathbb{E}\left[\overline{\mathcal{U}_{k}}^{H}\overline{\mathcal{U}_{k}}~ \overline{\mathcal{U}_{k}}^{H}\overline{\mathcal{U}_{k}} -\frac{2}{n_1}\overline{\mathcal{U}_{k}}^{H}\overline{\mathcal{U}_{k}} +\frac{1}{n_1^2}\bm{I}_{n_2n_3}\right]\right\| \\
   &=& \left\|\mathbb{E}\left[\overline{\mathcal{U}_{k}}^{H}\overline{\mathcal{U}_{k}} ~\overline{\mathcal{U}_{k}}^{H}\overline{\mathcal{U}_{k}}\right] -\frac{1}{n_1^2}\bm{I}_{n_2n_3}\right\|  \\
   &\leq & \max\left\{\left\|\mathbb{E}\left[\overline{\mathcal{U}_{k}}^{H}\overline{\mathcal{U}_{k}} ~\overline{\mathcal{U}_{k}}^{H}\overline{\mathcal{U}_{k}} \right]\right\|,\frac{1}{n_1^2}\right\} \\
   &\leq& \max\left\{c^2\frac{n_2\mu(\mathcal{S})}{n_1}\left\|\mathbb{E}\left[\overline{\mathcal{U}_{k}}^{H}\overline{\mathcal{U}_{k}} \right]\right\|_2,\frac{1}{n_1^2}\right\} \\
  &=& \max\left\{c^2\frac{n_2\mu(\mathcal{S})}{n_1^2}\left\|\bm{I}_{n_2n_3}\right\|,\frac{1}{n_1^2}\right\}\\
&=&\frac{c^2n_2\mu(\mathcal{S})}{n_1^2}.
\end{eqnarray*}
Let $\rho^2=\frac{c^2n_2\mu(\mathcal{S})}{n_1^2} $, and
by Bernstein Inequality of Lemma \ref{lem:mat}, we have
\begin{equation*}
  \mathbb{P} \left[\left\|\sum\limits_{i=1}^m\left(\overline{\mathcal{U}_{k}}^{H}\overline{\mathcal{U}_{k}} -\frac{1}{n_1}\bm{I}_{n_2n_3}\right)\right\|>\tau\right]
   \leq 2n_2n_3\exp\left(\frac{-\tau^2/2}{m\rho^2+M\tau/3}\right).
\end{equation*}
 We restrict $\tau$ to be $M\tau\leq m\rho^2$, then the equation can be simplified as
\begin{equation*}
  \mathbb{P} \left[\left\|\sum\limits_{i=1}^m\left(\overline{\mathcal{U}_{k}}^{H}\overline{\mathcal{U}_{k}} -\frac{1}{n_1}\bm{I}_{n_2n_3}\right)\right\|>\tau\right] \leq 2n_2n_3\exp\left(\frac{-3n_1^2\tau^2}{8mc^2n_2\mu(\mathcal{S})}\right).
\end{equation*}
Now set $\tau=\gamma m/n_1$ with $\gamma $ defined in the statement of Theorem \ref{theo:tubal}. Since $\gamma<1$ by assumption, $M\tau\leq m\rho^2$ holds. Then we have
\begin{equation*}
  \mathbb{P} \left[\left\|\sum\limits_{i=1}^m\left(\overline{\mathcal{U}_{k}}^{H}\overline{\mathcal{U}_{k}} -\frac{1}{n_1}\bm{I}_{n_2n_3}\right)\right\|\leq \frac{m}{n_1}\gamma\right] \geq 1-\delta .
\end{equation*}
We note that $\left\|\sum_{i=1}^m\left(\overline{\mathcal{U}_{k}}^{H}\overline{\mathcal{U}_{k}} -\frac{1}{n_1}\bm{I}_{n_2n_3}\right)\right\|_2\leq \frac{m}{n_1}\gamma$ implies that the minimum singular value of $\sum_{i=1}^m\left(\overline{\mathcal{U}_{k}}^{H}\overline{\mathcal{U}_{k}}\right)$ is at least $(1-\gamma)\frac{m}{n_1}$. This in turn implies that
\begin{equation*}
  \left\|\left(\sum\limits_{i=1}^m\overline{\mathcal{U}_{k}}^{H}\overline{\mathcal{U}_{k}} \right)^{-1}\right\|\leq \frac{n_1}{(1-\gamma)m}.
\end{equation*}
That means that
\begin{equation*}
 \left\|\left(\overline{\mathcal{U}_{\Omega}}^{H}\overline{\mathcal{U}_{\Omega}} \right)^{-1}\right\|\leq \frac{n_1}{(1-\gamma)m}
\end{equation*}
holds with the probability at least $1-\delta$.
\end{proof}

\begin{proof}[Proof of Lemma \ref{element:lemma2}]
We use Bernstein's inequality of Lemma \ref{lem:vec} to prove Lemma \ref{element:lemma2}. For elementwise-sampling, $\Omega\subset [n_1]\times [n_3] $. Then we set $\bm{x}_{k}=\bm{U}^H\left(:,(j-1)n_1+i\right)\bm{y}_{(j-1)n_1+i}\bm{1}_{\Omega(k)=(i,j)} $. Since $\bm{y}\in S^\bot $ and the samples are taken with replacement, we have the followings.
\begin{eqnarray*}
  \mathbb{E}\left[\bm{x}_{k}\right] &=& \mathbb{E}\left[\sum\limits_{p=1}^{n_1}\sum\limits_{q=1}^{n_3}\bm{U}^H\left(:,(q-1)n_1+p\right)\bm{y}_{(q-1)n_1+p}\bm{1}_{\Omega(k)=(p,q)}\right] \\
   &=& \frac{1}{n_1n_3}\sum\limits_{p=1}^{n_1}\sum\limits_{q=1}^{n_3}\bm{U}^H\left(:,(q-1)n_1+p\right)\bm{y}_{(q-1)n_1+p}=\bm{0}_{n_2n_3\times 1},
\end{eqnarray*}
\begin{eqnarray*}
  \sum\limits_{k=1}^m\mathbb{E}\left\|\bm{x}_{k}\right\|_2^2 &=& \sum\limits_{k=1}^m\mathbb{E}\left[\sum\limits_{p=1}^{n_1}\sum\limits_{q=1}^{n_3}\bm{U}^H\left(:,(q-1)n_1+p\right)\bm{y}_{(q-1)n_1+p}\bm{1}_{\Omega(k)=(p,q)}\right]\\
  &=& \frac{m}{n_1n_3}\sum\limits_{p=1}^{n_1}\sum\limits_{q=1}^{n_3}\|\bm{U}^H\left(:,(q-1)n_1+p\right)\bm{y}_{(q-1)n_1+p}\|_2^2 \\
   &=& \frac{m}{n_1n_3}\sum\limits_{p=1}^{n_1}\sum\limits_{q=1}^{n_3}\|\bm{U}^H\left(:,(q-1)n_1+p\right)\|_2^2\bm{y}_{(q-1)n_1+p}^2 \\
   &\leq& \frac{m}{n_1n_3}\frac{n_2\mu(\bm{S})}{n_1}\|\bm{y}\|_2^2 \\
   &=& \frac{mn_2\mu(\bm{S})}{n_1^2n_3}\|\bm{y}\|_2^2,
\end{eqnarray*}
We set $ M=\frac{mn_2\mu(\bm{S})}{n_1^2n_3}\|\bm{y}\|_2^2$ and $\tau=\sqrt{4 M\ln(1/\delta)} $. When $m\geq 4\frac{n_1n_3\|\bm{y}\|_{\infty}^2}{\|\bm{y}\|_2^2}\ln\left(\frac{1}{\delta}\right) $, $\sqrt{4 M\ln(1/\delta)}\leq M\left(\max\limits_{k}\|\bm{x}_{k}\|_2\right)^{-1} $. Then we have
\begin{eqnarray*}
  \left\|\bm{U}_{\Omega}\bm{y}_{\Omega}\right\|_2^2 &\leq& \left(\sqrt{ M}+\sqrt{4 M\ln(1/\delta)}\right)^2 \\
   &=& \left(1+2\sqrt{ln(1/\delta)}\right)^2 \frac{mn_2\mu(\bm{S})}{n_1^2n_3}\|\bm{y}\|_2^2 \\
   &=& \beta\frac{mn_2\mu(\bm{S})}{n_1^2n_3}\|\bm{y}\|_2^2
\end{eqnarray*}
with the probability at least $1-\delta$, where $\beta$ is defined in Theorem \ref{theo:element}.
\end{proof}

\end{document}